\documentclass[10pt]{amsart}

\usepackage{tikz}
\usepackage[all,cmtip]{xy}
\usetikzlibrary{calc}

\usepackage{ae} 
\usepackage[T1]{fontenc} 
\usepackage[cp1250]{inputenc}
\usepackage{amsmath}
\usepackage{amssymb, amsfonts,amscd,verbatim}

\usepackage{indentfirst}
\usepackage{latexsym}
\input xy
\xyoption{all}

\usepackage{amsmath}    

\theoremstyle{plain}
\newtheorem{Prop}{Proposition}[section]
\newtheorem{Thm}[Prop]{Theorem}
\newtheorem{Cor}[Prop]{Corollary}

\newtheorem{Lem}[Prop]{Lemma}

\theoremstyle{definition}
\newtheorem{Def}[Prop]{Definition}

\theoremstyle{remark}
\newtheorem{Rem}[Prop]{Remark}

\newcommand{\Int}{\operatorname{Int}\nolimits}

\newcommand{\diam}{\operatorname{Diam}\nolimits}

\newcommand{\Rips}{\operatorname{Rips}\nolimits}
\newcommand{\cRips}{\operatorname{\overline{Rips}}\nolimits}
\newcommand{\C}{\operatorname{Cech}\nolimits}
\newcommand{\cC}{\operatorname{\overline{Cech}}\nolimits}

\def\CC{{\mathbb C}}
\def\RR{{\mathbb R}}

\def\ZZ{{\mathbb Z}}
\def\NN{{\mathbb N}}

\def\cB{{\overline {B}}}
\def\cN{{\overline {N}}}

\def\di{{\partial}}
\def\eps{{\varepsilon}}
\def\s{\sigma}
\def\DC{\xrightarrow{DC}}
\def\SDC{\xrightarrow{SDC}}

\numberwithin{equation}{section}
\title[Footprints of geodesics in persistent homology]
{Footprints of geodesics in persistent homology}

\author{\v Ziga ~Virk}
\address{University of Ljubljana}
\email{ziga.virk@fri.uni-lj.si}
\thanks{Research was  supported by Slovenian Research Agency grants No. N1-0114 and P1-0292. }

\begin{document}

\maketitle
\begin{center}
\today
\end{center}

\begin{abstract}
Given a metric space $X$ and a  subspace $A\subset X$, we prove $A$ can generate various algebraic elements in persistent homology of $X$. We call such elements (algebraic) footprints of $A$. Our results imply that footprints typically appear in dimensions above $\dim(A)$. Higher-dimensional persistent homology thus encodes lower-dimensional geometric features of $X$.

We pay special attention to a specific type of geodesics  in a geodesic surface $X$ called geodesic circles.  We explain how they may generate non-trivial odd-dimensional and two-dimensional footprints. In particular, we can detect even some contractible  geodesics using two- and three-dimensional persistent homology.
This provides a link between persistent homology and the length spectrum in Riemannian geometry.
\end{abstract}



\section{Introduction}
\label{SectIntro}

Algebraic topology is based on the notions of homotopy and homology groups. These invariants have  proved invaluable in most parts of mathematics. They were designed to detect holes of certain dimension, either through maps from spheres or via (co)chains. By construction, the $k$-dimensional homology $H_k$ of an $n$-dimensional complex is trivial for all $k>n$. The same was assumed to be true for homotopy groups, until the famous discovery \cite{Hopf} of the Hopf map, which demonstrated that $\pi_3(S^2)$ is non-trivial. This result completely transformed the perception of the homotopy groups, had profound consequences in mathematics, and launched an extensive research effort to understand the underlying machinery. 

At the turn of the century persistent homology appeared in the setting of computational topology, motivated in part by problems of data analysis of geometric shapes. The idea was to obtain a theoretical and computational framework, which would detect and measure sizes of holes. The resulting persistent homology has since had a profound impact on the development of theory and on applications in other sciences. At the heart of this approach is a filtration of simplicial complexes, often arising from \v Cech or Rips filtrations via increasing scale parameters. As decades before in algebraic topology, the intuition is that features of persistent homology in dimension $k$ correspond to $k$-dimensional holes in the shape. This intuition is well justified by the Nerve Theorem if our underlying space lies in the Euclidean space and we use the Euclidean distance to build the \v Cech filtration. However, it is known \cite{CS} that in the setting of Rips complexes such expectation is overly optimistic. Furthermore, the ambient Euclidean metric is usually chosen for reasons of computational simplicity. In many cases it might be more natural to consider a geodesic metric on the underlying space, which is the basic structure of Riemannian geometry, or perhaps some other statistically motivated distances \cite{EW} (for example the relative entropy, which is not even a metric; or the Fisher metric, which is in fact a geodesic metric).

In this paper we will focus on persistent homology via \v Cech and Rips filtrations of geodesic spaces. Our starting point is a surprising deep Theorem \ref{ThmAA}, recently proved in \cite{AA}. The theorem states that the mentioned open filtrations of geodesic $S^1$ first attain the homotopy type of $S^1$, then of $S^3$, $S^5$, etc. The theorem thus plays a role of the Hopf map, demonstrating that $1$-dimensional space $S^1$ generates non-trivial footprints in persistent higher-dimensional homology. Consequently it carries the same lesson for the perception of persistent homology as the original Hopf map had for the algebraic topology: non-trivial elements  in persistent homology in dimension $k$ may be generated by features of dimension less than $k$.

In this paper we intend to make use of  this Hopf effect to extract new geometric information about the underlying geodesic space $X$ from persistent homology:
\begin{itemize}
 \item First we introduce a general framework for footprint detection (Theorems \ref{ThmMain1} and \ref{ThmMain15}), i.e., a result stating that for suitable subspaces $A\subset X$, a part of the persistent homology of $A$ is contained in the persistent homology of $X$.
 \item  We then prove Theorem \ref{ThmMain2}, stating that for a loop $A$ on a surface $X$, a footprint of $A$ may sometimes contain a $2$-dimensional interval not contained in the persistent homology of $A$. 
\end{itemize}
In particular we show how certain loops (appropriately isometrically embedded geodesic circles of positive circumference, regular polygons and some ellipses) can be detected using persistent homology, even if they are contractible. 
For a graphical example of these results see Figure \ref{FigEssence}. An application of these results is demonstrated in Section \ref{SectSample}.

Our eventual far-reaching goal, which is probably unattainable in its full generality, is understanding the correspondence between geometric features of a spaces and algebraic elements of persistent homology. We would like to know which geometric features generate footprints in persistent homology and how to recreate them. This paper establishes such a correspondence for nicely embedded geodesic circles. There are now three different types of footprints that might be induced by a geodesic circle on a geodesic surface:
\begin{itemize}
 \item A $1$-dimensional \textbf{topological footprint} appearing  if the loop is a member of a lexicographically shortest homology basis \cite{ZV}.
 \item A $3$-dimensional (and higher-dimensional) \textbf{combinatorial footprint} arising from the internal combinatorics of the Rips complex of a circle as described by Theorem \ref{ThmMain1}.
 \item A $2$-dimensional \textbf{geometric footprint} appearing under certain local geometric conditions at the loop in the absence of a topological footprint as described by Theorem \ref{ThmMain2}.
\end{itemize}

Further footprint detection results could be obtained if we could compute the persistent homology of simple spaces, such as spheres, etc.

The results of this paper have further theoretical and practical consequences,  some of which we plan to explore in the future work:
\begin{enumerate}
\item By results of \cite{ZV}, the critical values of the persistent fundamental group $\{\pi_1(\Rips(X,\bullet),r)\}_{r>0}$, which are in general incomputable due to the word problem in groups,  correspond to geodesic circles. Results of this paper indicate that candidates for these critical values can in some cases be extracted from higher dimensional homology. 
\item The collection of lengths of closed geodesics features prominently in differential geometry under the name of the length spectrum (for a modern treatment see \cite[Section 7.2]{LengthSpec} or \cite{LS2}). It is closely related to the Laplacian spectrum and in some cases, even to the volume of the manifold. The results of this paper detect a part of this spectrum arising from geodesic circles. By extending our method we hope establish a result describing how much of the length spectrum is encoded in the persistent homology. 
\item As was already mentioned in \cite{ZV}, the setting of geodesic spaces provides a convenient venue for topological data analysis via persistent homology for a number of reasons: filtrations are smaller, they seem to be more stable \cite{ZV1} and seem to contain less noise (see example of Section \ref{SectSample}), there seems to be an inherent structure to the corresponding persistence diagram, etc.
\end{enumerate}

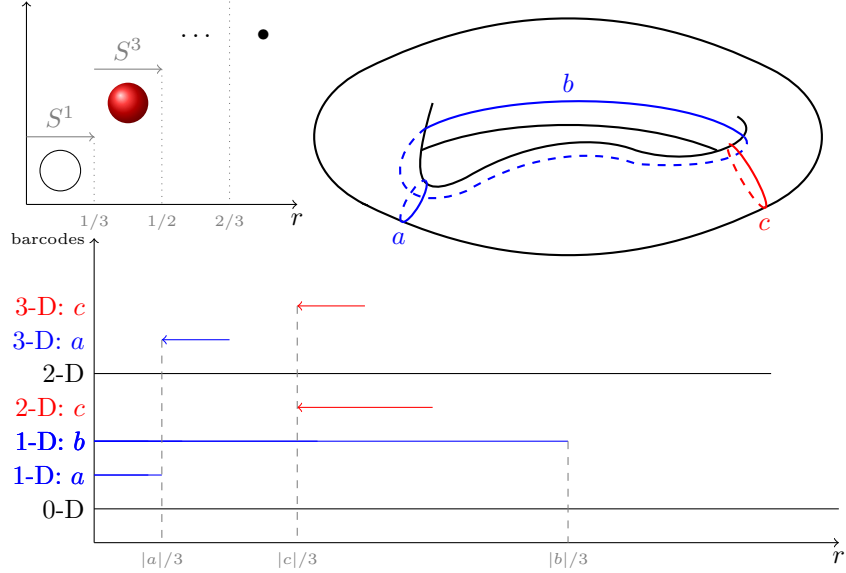
\begin{figure}
\begin{tikzpicture}[scale=.9]
\draw [thick] (-3,0) ..controls (-1,-1)and(1,-1)..
(3,0)..controls(4,.5)and(4, 1.5)..
(3,2)..controls(1,3)and(-1, 3)..
(-3,2)..controls(-4,1.5)and(-4, .5)..
cycle;
\draw [thick](-2,1.5) ..controls (-2.3,.5)and(-2.3, 0)..
(-1.5, .4)..controls(-.5,1) and(.5,1)..
(1,.8)..controls(2, 0.5)and(3, 1)..(2.5,1.3);
\draw  [thick](-2.17,.8) ..controls (-1,1.3)and(1,1.3)..(2.2,.8);
\draw [blue, thick, rotate=-30](-2,-1.45) arc (-90:90:.1 and .35);
\draw [blue, thick, dashed, rotate=-30](-2,-1.45) arc (-90:-270:.1 and .35);
\draw [blue](-2.5, -.5) node {$a$};
\draw [blue, thick](2.6,1) arc (20:150:2.6 and .8);
\draw [thick, blue, dashed](-2.1,1.13) ..controls (-2.8,.6)and(-2.5, -.2)..
(-1.5, 0.2)..controls(-.5,.8) and(.5,.8)..
(1,.6)..controls(2.3, 0.5)and(2.8, .8)..(2.6,1.0);
\draw [blue](0, 1.8) node {$b$};
\draw [red, thick, rotate=30](2.5,-1.5) arc (-90:90:.1 and .54);
\draw [red, thick, dashed, rotate=30](2.5,-1.5) arc (-90:-270:.1 and .55);
\draw [red](2.9, -.3) node {$c$};
\draw[->] (-8, 0) to (-4, 0) node[below]{$r$};
\draw[->] (-8, 0) to (-8, 3) node{};
\draw[->] [gray] (-8,1) to node [above]{$S^1$}(-7, 1);
\draw[dotted, gray]  (-7,1) to (-7, -0) node[below]{\tiny{$1/3$}};
\draw[->] [gray] (-7,2) to node [above]{$S^3$}(-6, 2);
\draw[dotted, gray]  (-6,2) to (-6, 0) node[below]{\tiny{$1/2$}};
\draw (-5.5, 2.5) node {$\ldots$};
\draw[dotted, gray]  (-5,3) to (-5, 0) node[below]{\tiny{$2/3$}};
\draw (-4.5, 2.5) node {$\bullet$};
\draw  (-7.5,.5) circle [radius=0.3];
\shade [ball color=red]  (-6.5,1.5) circle (0.3);
\draw [->] (-7, -5) to (4, -5) node[below] {$r$};
\draw [->] (-7, -5) to (-7, -.5) node[ left] {\tiny barcodes};
\draw (-7, -4.5) node[left]{$0$-D} to (4, -4.5);
\draw [blue] (-7, -4)node[left]{ $1$-D: \color{blue} $a$}  to (-6, -4);

\draw [ blue] (-7, -4)node[left]{ $1$-D: \color{blue} $a$} to (-6.2, -4);

\draw [blue] (-7, -3.5)node[left]{ $1$-D: \color{blue} $b$} to (0, -3.5);

\draw [blue] (-7, -3.5)node[left]{ $1$-D: \color{blue} $b$} to (-3.7, -3.5);
\draw [blue] (-7, -3.5)node[left]{ $1$-D: \color{blue} $b$} to (-5.7, -3.5);
\draw [blue] (-7, -3.5)node[left]{ $1$-D: \color{blue} $b$} to (-6.2, -3.5);


\draw (-7, -3)node[left]{\color{red} $2$-D:  $c$};
\draw [<-,red] (-4, -3) to (-2, -3);

\draw  (-7, -2.5)node[left]{ $2$-D} to (3, -2.5) node [right] {};

\draw  [blue](-7, -2)node[left]{ $3$-D: \color{blue} $a$};
\draw  [<-,blue](-6, -2) to (-5, -2) node [right] {};


\draw (-7, -1.5)node[left]{\color{red} $3$-D:  $c$};
\draw [<-, red] (-4, -1.5)to (-3, -1.5) node [right] {};


\draw[dashed, gray](-6, -2) to (-6, -5) node[below]{\tiny$|a|/3$};


\draw[dashed, gray](-4, -1.5) to (-4, -5) node[below]{\tiny$|c|/3$};
\draw[dashed, gray](0, -3.5) to (0, -5) node[below]{\tiny$|b|/3$};

\end{tikzpicture}
\caption{An example of a footprint detection. The upper left side represents the homotopy type of the Vietoris-Rips filtration of a circle (i.e., odd-dimensional spheres) equipped with a geodesic metric by Theorem \ref{ThmAA}. The black shape on the right is a two-dimensional torus and below it is an excerpt from its barcode as retrieved from our results. Geodesic circle $a$ is a member of the shortest homology base and  hence generates odd-dimensional spheres and bars by Theorem \ref{ThmMain1}. Geodesic circle $c$ is not a member of such a base, hence it generates higher odd-dimensional bars along with a $2$-dimensional bar
by Theorem \ref{ThmMain2}.}
\label{FigEssence}
\end{figure}

\textbf{Related work}: 
One of the first appearances of filtrations goes back to the introduction of \v Cech (co)homology. The approach later evolved into the Shape Theory \cite{DySe78}, which studies the limiting behaviour as $r \to 0$. In the analogous setting, such approximations with $r \to \infty$ have been employed in asymptotic topology \cite{Dra02} and \cite{Comb}. Reconstructing the homotopy type of a manifold for small $r$ was considered in \cite{Haus}. Similar reconstructions in the geodesic setting were considered in \cite{Lat} and are in general a subject of study in computational topology \cite{W, Att}. The study of filtrations for all values of $r$ in the geodesic setting started with the case of $S^{1}$ in \cite{AA}. The $1$-dimensional persistence was first considered in \cite{7A} for metric graphs, and completely developed in \cite{ZV, ZV1}. Further results that can be used in conjunction with our footprint detection procedure, contain ellipses \cite{AAS} and regular polygons \cite{A3}. A connection between persistent homology and some geometric notions in simplicial complexes has recently been treated in \cite{ACos}. A connection to the filling radius has been established in \cite{Memoli}.

\
 
The structure of the paper is the following. In Section \ref{SectPre} we provide preliminaries. Section \ref{SectDefCon} introduces deformation contractions as our  tool. Sections \ref{SectFirs} and \ref{SectSec} provide a  footprint detection framework in a specific and general setting respectively. Section \ref{SectNull} describes the combinatorics of nullhomologies of loops. Section \ref{Sect2D} is the most technical section, describing an emergence of a two-dimensional footprint. Section \ref{SectCech} extends the results to \v Cech complexes and closed filtrations. Section \ref{SectSample} provides an example of the interpretation using the results of this paper.

\section{Preliminaries}
\label{SectPre}

Let $(X,d)$ be a metric space and $r>0$. For $x\in X$ let  $B(x,r)$ and $\cB(x,r)$ denote the open and closed balls centred at $x$ of radius $r$. For $A\subset  X$, notation $N_X(A,r)$ (or $N(A,r)$, when it is clear what the ambient space is) represents the open neighborhood, and $\cN(A,r)$ represents the closed neighborhood around $A$ of radius $r$ in $X$. 

A space $X$ is \textbf{geodesic}, if for each distinct $ x,y \in X$ there exists an isometric embedding $g\colon [0,d(x,y)] \to X$ with $g(0)=x$ and $g(d(x,y))=y$, i.e., if $x$ and $y$ are connected by a path of length $|g|=d(x,y)$, which is called a \textbf{geodesic} (in the literature, the notion of a geodesic sometimes refers to what we would call a local geodesic, which differs from our notion of geodesic). A \textbf{geodesic surface} is a surface equipped with a  geodesic metric. A subset $A$ of a geodesic space $X$ is \textbf{geodesically convex}, if for each $x,y\in A$, every geodesic between $x$ and $y$ in $X$ is contained in $A$. A \textbf{geodesic circle} in $X$ is an isometrically embedded circle $(S, d_S) \hookrightarrow (X,d_X)$, where $d_S$ is a geodesic metric. If $G$ is an Abelian group and $\alpha$ is a loop in $X$ then $[\alpha]_G \in H_1(X,G)$ is the homology element represented by $\alpha$.

We next define the height of  homotopy, which is motivated by the combinatorial version of the height in \cite{CL}. 
Given a homotopy $H \colon S^1 \times I\to X$ between loops $H_{S^1\times \{0\}}$ and $H_{S^1\times \{1\}}$, its \textbf{height} is the length of the longest intermediate curve $H_{S^1\times \{t\}}$. A \textbf{homotopy height} between homotopic loops is the infimum of the heights of all homotopies between the loops. Similarly, a nullhomotopy height of a contractible loop is  the infimum of the heights of all nullhomotopies.

Given  $r>0$ we define various simplicial complexes with the vertex set $
X$. For a longer discussion on the subject see \cite{ZV2}. The condition next to the name determines when a finite subset $\s\subset X$ belongs to the complex.
\begin{enumerate}
 \item \textbf{(Open) Rips} (or Vietoris-Rips) \textbf{complex} $\Rips(X,r): Diam(\s) < r$.
 \item \textbf{Closed Rips complex} $\cRips(X,r): Diam(\s) \leq r$.
 \item \textbf{(Open) \v Cech complex} $\C(X,r): \cap_{z\in\s} B(z,r) \neq \emptyset$.
  \item \textbf{Closed \v Cech complex} $\cC(X,r): \cap_{z\in\s} \cB(z,r) \neq \emptyset$.
\end{enumerate}

When considering \v Cech complexes of subsets $A\subset X$ we need to specify where we look for an intersection. If the context is not clear, then $\C_X(A,r)$ consists of finite subsets $\sigma \subset A$ for which $\cap_{z\in\s} B_X(z,r) \neq \emptyset$ in $X$; similarly, $\C_A(A,r)$ consists of finite subsets $\sigma \subset A$ for which $\cap_{z\in\s} B_A(z,r) \neq \emptyset$ in $A$. The same goes for the closed \v Cech complexes. Note that if $\alpha$ is a geodesic circle in $X$, then $\C_X(\alpha,r)=\C_\alpha(\alpha, r)$ for all $r>0$.

We will refer to  $1$-dimensional simplices as edges and to $2$-dimensional simplices as triangles. For a simplex $\sigma$ in a \v Cech complex we refer to any point $w$ of $\cap_{z\in\s} B(z,r)$ as a witness of $\sigma$, or say that $w$ witnesses $\sigma$.

Given any complex $\mathcal{C}$ mentioned above, we construct a corresponding \textbf{filtration} $\{\mathcal{C}(X,r)\}_{r>0}$ as a collection of complexes $\mathcal{C}(X,r)$ for all positive parameters, bound together by the bonding inclusions $i_{p,q}\colon \mathcal{C}(X,p)\hookrightarrow \mathcal{C}(X,q)$, which are identities on the vertices for all $p<q$. Two filtrations $\{A_r\}_{r>0}$ and $\{B_r\}_{r>0}$ are \textbf{homotopy equivalent}, if there exists a homotopy equivalence between them, i.e., a collection of homotopy equivalences $\{f_r\colon A_r \to B_r\}_{r>0}$, commuting with the bonding maps up to homotopy.

The following is the main result of 
\cite{AA}, which will be crucial in our arguments. It describes the homotopy types of filtrations of a circle equipped with a geodesic metric.

\begin{Thm}\cite{AA}
\label{ThmAA}
Suppose $S$ is a circle, equipped with a geodesic metric so that it is of circumference $1$. Then for all $l=0,1,2,\ldots$
$$
\Rips(S, r) \simeq S^{2l + 1}, \qquad \textit{ for } \frac{l}{2l+1}<r \leq\frac{l+1}{2l+3},
$$
$$
\C(S, r) \simeq S^{2l + 1}, \qquad \textit{ for } \frac{l}{2(l+1)}<r \leq \frac{l+1}{2(l+2)},
$$
$$
\cRips(S,r)\simeq 
\begin{cases}
   S^{2l + 1}, \qquad \textit{ for } \frac{l}{2l+1}<r<\frac{l+1}{2l+3},\\
 \bigvee^{|\RR|} S^{2l}, \qquad \textit{ for } r=\frac{l}{2l+1},
\end{cases}
$$
$$
\cC(S,r)\simeq 
\begin{cases}
   S^{2l + 1}, \qquad \textit{ for } \frac{l}{2(l+1)}<r<\frac{l+1}{2(l+2)},\\
 \bigvee^{|\RR|} S^{2l}, \qquad \textit{ for } r=\frac{l}{2(l+1)}.
\end{cases}
$$
Furthermore, the bonding maps on filtrations are homotopy equivalences whenever possible.

For $r\geq 1/2$, all the mentioned complexes are contractible. 
\end{Thm}

For a loop (or a path) $\alpha\colon I \to X$ and $r>0$, an $r$-\textbf{sample} of $\alpha$ is a sequence $\alpha(t_0), \alpha(t_1), \alpha(t_2), \ldots, \alpha(t_k)$, where $t_0=0< t_1 < \ldots < t_k=1$ and for each $i$ $\diam(\alpha|_{[t_i, t_{i+1}]}) < r$ holds. See \cite[Section 3]{ZV} for details and properties of $r$-samples. We will often identify an $r$-loop with the simplicial cycle in $\Rips(X,r)$ consisting of edges $[t_i, t_{i+1}]$. We will often utilise a transition from the continuous setting of $X$ to the discrete setting of Rips and \v Cech complexes as introduced in \cite{ZV} via $r$-samples of loops.

By \textbf{persistence} we mean any object, obtained by applying any  homology group functor to any filtration. For example, a $H_1(\_, \ZZ)$ persistence via the closed \v Cech filtration of $X$ is a collection $\{H_1(\cC(X,r),\ZZ)\}_{r>0}$ along with the induced bonding maps. In the paper we will sometimes consider restrictions of parameter $r$. We will often be using relations (such as isomorphisms) and operations (such as direct sums) on such persistences: such operations will always consist of level-wise maps, which are  consistent (i.e., commutative) with the bonding maps. Where there is no ambiguity about the coefficients, such as in the proofs, we will omit them from the notation of homology. For an Abelian group $G$, the maps induced by the bonding maps $i_{p,q}$ on homology with coefficients in $G$ are denoted by $i^{G}_{p,q}$.
Given $\eps>0$, an $\eps$-\textbf{interleaving} between two filtrations $\{A_r\}_{r>0}$ and $\{B_r\}_{r>0}$ consists of collections of maps $\{f_r\colon A_r \to B_{r+\eps}\}_{r>0}$ and $\{g_r\colon B_r \to A_{r+\eps}\}_{r>0}$ that commute with the bonding maps. 

For $b \leq d$ the notation $\langle b,d \rangle$ represents an interval. We use this notation when we do not want to commit to a specific type of endpoints of an interval. In particular, $\langle b,d \rangle$ can be either $(b,d)$ or $[b, d)$, etc.

For the rest of this section let us assume $G$ is a field.   Given an interval $\langle b,d \rangle \subseteq (0, \infty)$, $G_{\langle b,d \rangle}$, is a collection of vector spaces
$\{ U_r \}_{r>0}$ defined by
\begin{align}
  U_r  &=  \begin{cases} G, & r \in    \langle b,d \rangle \\
                         0,  & r \notin \langle b,d \rangle
           \end{cases}
\end{align}
and by setting all bonding maps $U_{r} \to U_{s}$, with $r, s \in \langle b,d\rangle$,
to be isomorphisms. These are called (elementary) \textbf{interval modules}, intervals, or just bars.  For compact spaces $X$, the $k$-dimensional $H_k(\_,G)$ persistence decomposes as a direct sum of such bars, which together constitute  a barcode. A decomposition also exists for any persistence obtained through open \v Cech or Rips filtration (see q-tameness condition in Proposition 5.1 of \cite{Cha2} and the property of being radical in \cite{ChaObs} for details). A \textbf{persistence diagram} PD is an alternative description of a barcode. It consists of a collection of points in a plane, one point corresponding to each such bar, with the coordinates of a point being the left endpoint (birth) and the right endpoint (death) of the corresponding bar. 

%

Given a compact geodesic locally contractible space, there exists (see Definition 8.7 and Proposition 8.9 of \cite{ZV} for details) a \textbf{lexicographically minimal basis} of $H_1(X,G)$ consisting of a finite collection of geodesic circles $a_1, a_2, \ldots, a_k$ with $|a_1|\leq |a_2| \leq \ldots \leq |a_k|$, such that:
\begin{enumerate}
 \item for each $i$  homology class $[a_i]$ is not an element of the subgroup of $H_1(X,G)$ generated by $[a_1],[a_2], \ldots, [a_{i-1}]$, and
  \item homology classes $[a_1], [a_2],\ldots, [a_k]$ generate $H_1(X,G)$.
\end{enumerate}
 The lexicographical minimality refers to the fact that if $b_1, b_2, \ldots, b_m$ with $|b_1|\leq |b_2| \leq \ldots 
\leq |b_m|$ is another collection of loops in $X$ satisfying (1) and (2) above and with $[|b_1|, |b_2|, \ldots, |b_m|]$ being lexicographically smaller than  $([a_1],[a_2], \ldots, [a_{k}])$, then $m=k$ and $|b_i|= |a_i|, \forall i$.

\begin{Thm} \cite{ZV}
 \label{Thm1Dim}
 Given a compact geodesic locally contractible space, a field $G$, and a lexicographically minimal basis of $H_1(X,G)$ consisting of a finite collection of geodesic circles $a_1, a_2, \ldots, a_k$ with $|a_1|\leq |a_2| \leq \ldots \leq |a_k|$, persistence $\{H_1(\Rips(X,r),G)\}_{r>0}$ is isomorphic to the direct sum of interval modules $G_{( 0,|a_1|/3 ]}$, $G_{( 0,|a_2|/3 ]}$, \ldots, $G_{( 0,|a_k|/3 ]}$.
\end{Thm}

\section{Deformation contractions}
\label{SectDefCon}

Crushings were first defined in \cite{Haus} as a type of maps inducing homotopy equivalences on the corresponding Rips complexes. In this paper we will refer to them as deformation contractions and prove they also induce homotopy equivalences on the corresponding \v Cech complexes. In our context,  deformation contractions are the crucial tool connecting metric properties of a space to the homotopy properties of the corresponding Rips or \v Cech complexes.

\begin{Def}\cite{Haus}\label{DefDC}
Let $X$ be a metric space and $A\subset X$. A continuous map $F\colon X \times [0,1] \to X$ is called a \textbf{deformation contraction} (we will abbreviate it as DC and write $X \DC A$) if:
\begin{enumerate}
 \item $F(x,0)=x, F(x,1)\in A, F(a,t)=a, \forall x\in X, a\in A, t\in [0,1]$, and
 \item $d(F(x,t'),F(y,t')) \leq d(F(x,t),F(y,t)), \forall x,y\in X, t'>t$.
\end{enumerate}
If additionally $d(F(x,t'),F(y,t')) < d(F(x,t),F(y,t))$ holds for all pairs $  (x,y)\in (X\setminus A) \times X$ with $x \neq y$ and for all $ t'>t$, then $F$ is called a strict deformation contraction (SDC or $X \SDC A$).
\end{Def}

It is easy to see that if $X$ is geodesic and $X \SDC A$, then $A$ is geodesically convex in  $X$. Furthermore, if $X$ is geodesic and $X \DC A$, then $A$ equipped with the subspace metric is a geodesic space. If $N(A,r) \DC A$, where $A\subset X$ is a subspace, then $\C_A(A,r)=\C_X(A,r)$. We will  be using this last fact generously throughout the paper whenever \v Cech complexes will be involved.

Proposition \ref{PropDC} for Rips complexes was first proved in \cite{Haus}. Here we present an adaptation of that proof to the case of \v Cech complexes.

\begin{Prop}
 \label{PropDC}
 Suppose $X \DC A$. Then the inclusions $\Rips(A,r)\hookrightarrow \Rips(X,r)$ and $\C(A,r)\hookrightarrow \C(X,r)$ are homotopy equivalences for each $r>0$. 
\end{Prop}

\begin{proof}
 As mentioned above, we will only prove the \v Cech case by adjusting Hausmann's proof.
 
We will prove the following claim: for each pair of finite simplicial complexes $K_0 \leq K$, each simplicial map 
$$
g\colon (K, K_0) \to (\C(X,r), \C(A,r))
$$
is homotopic (rel $K_0$) to a simplicial map $f\colon (K, K_0) \to (\C(X,r), \C(A,r))$ with $f(K)\subset \C(A,r)$. 
The claim implies that the induced maps on the homotopy groups are isomorphisms and the conclusion of the proposition follows from the Whitehead theorem as in \cite{Haus} (see \cite{Hat} for the necessary background in algebraic toplogy).

Define landmarks $L=g(K^{(0)})\subset X$ and for each simplex $\sigma \in K$ choose a witness $w_\sigma \in \cap_{z\in \sigma} B(z,r)\subset X$. Choose also some $\eps>0$ with   $\eps < r-\max_{\sigma\in K} \max_{z\in \sigma} d(z,w_\sigma)$. At last choose $p\in \NN$ so that for all $k\in \{0, 1, \ldots, p-1\}$ and for all $x\in L$, 
$$
d\Bigg(F\Bigg(x, \frac{k}{p}\Bigg), F\Bigg(x, \frac{k+1}{p}\Bigg)\Bigg)<\eps,
$$
where $F$ is the deformation contraction given in the hypotheses of the proposition.
For each $k\in \{0, 1, \ldots, p\}$, a rule $z\mapsto  F\big(g(z), \frac{k}{p}\big)$ mapping $K^{(0)} \to X$ induces a simplicial map $f^k \colon K \to \C(X,r)$. Note that for  each simplex $\sigma \in K$, simplex $f^{k}(\sigma)$ is witnessed by $F\big(w_\sigma, \frac{k}{p}\big)$ by the property of DC. Furthermore, each $f^k$ is constant on $K_0$. 

Choose $\sigma\in K$. Note that for each $k\in \{0, 1, \ldots, p-1\}$ and $z\in \sigma$ we have 
$$
d(f^{k+1}(z), F(w_\sigma, k/p)) \leq d(f^{k+1}(z), f^{k}(z)) + d(f^{k}(z), F(w_\sigma, k/p)) < \eps +(r-\eps) =r,
$$
hence $F(w_\sigma, p/k)$ witnesses a simplex in $\C(X,r)$ containing $f^k(\sigma)$ and $f^{k+1}(\sigma)$. Consequently, $f^k$ and $f^{k+1}$ are contiguous rel $K_0$, hence homotopic rel $K_0$. Inductively we conclude that $g=f^0$ and $f=f^p$ are homotopic rel $K_0$, which proves our claim since $f(K^{(0)})\subset A$.
\end{proof}

\begin{Cor}
 \label{CorDCEquiv}
 Suppose $X \DC A$. Then open Rips filtrations of $X$ and $A$ are homotopy equivalent, and open \v Cech filtrations of $X$ and $A$ are homotopy equivalent.
\end{Cor}

\begin{proof}
 The inclusions used in Proposition \ref{PropDC} and its counterpart in \cite{Haus} obviously commute with the bonding inclusions of the filtrations.
\end{proof}

\begin{Rem}
\label{RemException}
 Proposition \ref{PropDC} does not hold for closed filtrations. Consider $X=[0,1] \times \{0,1\} \subset \RR^2$ in the Euclidean metric. It is easy to see that it deformation contracts to $A=\{0\}\times \{0,1\} \subset \RR^2$. However, $\cRips(X,1)$ has uncountable fundamental group, while $\cRips(A,1)$ is an edge.
 Proposition \ref{PropDC} does not hold for closed filtrations and SDC either, see Figure \ref{FigClosedEx}.
\end{Rem}

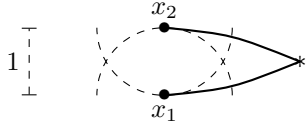
\begin{figure}

\begin{tikzpicture}[scale=.9]
\draw[thin, dashed, |-|] (0,0)-- node[left]{$1$}(0,1);
\draw [thin, dashed] (3,0) arc (0:180:1);
\draw[thin, dashed]  (3,1) arc (0:-180:1);
\draw[thick] (2,0) node{$\bullet$}node[below]{$x_1$}..controls(3,.1)..(4, .5) node{$\ast$}..controls(3,.9)..(2,1)node{$\bullet$}node[above]{$x_2$};

\end{tikzpicture}
\caption{A demonstration that a SDC may not induce homotopy equivalence on closed filtrations. The underlying planar space $X$ (in bold) consists of two points (bullets $x_1$ and $x_2$) at distance $1$, and two corresponding arcs connecting bullets to a point $\ast$ far away. The intersection of $\cB(x_1,1)$ and the opposite arc is precisely $x_2$, and vice versa. This implies   that $[x_1,x_2]$ is a maximal simplex in $\cRips(X,1)$ and thus $\pi_1(\cRips(X,1),\ast)\neq 1$. However, if the $y$-coordinates of arcs are changing monotonically, then $X \SDC \ast$ and $\cRips(\ast,1)$ is just a point.}
\label{FigClosedEx}
\end{figure}

When considering deformation contraction of subspaces, we can use Proposition \ref{PropDC} to obtain induced maps on the Rips complexes. However, we do have to be a bit more careful when considering \v Cech complexes.

\begin{Prop}
\label{PropCDC}
 Suppose $A\subset B \subset X$, $r>0$, and $N_X(B, r) \DC A$. Then $\C_X(A,r)\simeq \C_X(B,r)$.
\end{Prop}

\begin{proof}
 The proof of Proposition \ref{PropDC} applies as witnesses of simplices of $\C_X(B,r)$ are contained in $N_X(B, r)$.
\end{proof}

\section{Basic example of footprint detection}
\label{SectFirs}

In this section we present the prototype of a footprint detection. We use deformation contractions and the Mayer-Vietoris sequence on a surface to locally extract a footprint of a geodesic circle. More general conditions of this technique are provided in the subsequent section.

In this section we assume $X$ is a geodesic surface and  a geodesic circle $\alpha$ has some orientable subsurface as a neighborhood.

\begin{Def}
 \label{DefN1}
Suppose $0 < D_1 \leq D_2$. A loop $\alpha\subset X$ is $DC(D_1, D_2)$ isolated (deformation contraction isolated) if there exist two closed nested  neighborhoods $N_1 \subset N_2$ of $\alpha$, so that:
\begin{enumerate}
 \item $N_1$ and  $N_2$ are homeomorphic to $S^1 \times [0,1]$;
 \item $N_2 \supset N(N_1, D_2)$;
 \item $\di N_1$ consists of loops $\alpha_1$ and $\alpha_2$, which are at least $D_1$ apart from each other;
 \item $\Rips(\alpha_i, r)\simeq S^1, \forall i, \forall r<D_1$;
 \item $N_2 \setminus \Int(N_1) \DC \di N_1$ and $N_1 \DC \alpha$.
\end{enumerate}

See Figure \ref{FigDef} for a sketch.

Loop $\alpha$ is $DC(D)$ isolated, if it is $DC(D,D)$ isolated. 
Loop $\alpha$ is $SDC(D)$ isolated or $SDC(D_1, D_2)$ isolated, if furthermore all deformation contractions involved are strict deformation contractions. 
\end{Def}

\begin{figure}
\begin{tikzpicture}[scale=.9]
\draw [thick](-2.8,.9) ..controls (0,-.5)..(2.8,.9);
\draw [thick](-2.8,-1.7) ..controls (0,-.4)..(2.8,-1.7);
\draw [red, thick](0,-.75) arc (-90:90:.1 and .3);
\draw [red, thick, dashed](0,-.15) arc (90:270:.1 and .3);
\draw [red](0, -.75) node [below]{$\alpha$};
\draw [blue, thick](-1.3,-1.) arc (-90:90:.2 and .6);
\draw [blue, thick, dashed](-1.3,.2) arc (90:270:.2 and .6);
\draw [blue](-1.3, -1.05) node [below]{$\alpha_1$};
\draw [blue, thick](1.3,-1.) arc (-90:90:.2 and .6);
\draw [blue, thick, dashed](1.3,.2) arc (90:270:.2 and .6);
\draw [blue](1.3, -1.05) node [below]{$\alpha_2$};
\draw [ thick](-2.8,-1.7) arc (-70:70:.5 and 1.4);
\draw [ thick](-2.8,.9) arc (110:250:.4 and 1.4);
\draw [ thick](2.8,-1.7) arc (-70:70:.5 and 1.4);
\draw [ thick, dashed](2.8,.9) arc (110:250:.4 and 1.4);
\end{tikzpicture}

\begin{tabular}{cc}
     \begin{tikzpicture}[scale=.9]
\filldraw [draw=black,bottom color=black!30, top color=white, thick]
(-2.8,.9) ..controls (0,-.5)..(2.8,.9)
arc (70:-70:.4 and 1.4)
..controls (0,-.4)..(-2.8,-1.7)
arc (-70:70:.5 and 1.4);
\draw [ thick](-2.8,.9) arc (110:250:.4 and 1.4);
\draw [ dashed](2.8,.9) arc (110:250:.4 and 1.4);
\draw (0,-.75) arc (-90:90:.1 and .3);
\draw [dashed](0,-.15) arc (90:270:.1 and .3);
\draw (-1.3,-1.) arc (-90:90:.2 and .6);
\draw [dashed](-1.3,.2) arc (90:270:.2 and .6);
\draw (1.3,-1.) arc (-90:90:.2 and .6);
\draw [dashed](1.3,.2) arc (90:270:.2 and .6);
\draw[->, very thick] (-2.4,-.5) -- (-1.7,-.5);
\draw[->, very thick] (2.4,-.5) -- (1.7,-.5);
\end{tikzpicture}

&

\begin{tikzpicture}[scale=.9]
\filldraw [draw=black,bottom color=black!30, top color=white, thick]
(-1.3,.19) ..controls (0,-.29)..(1.3,.19)
arc (90:-90:.2 and .6)
..controls (0,-.63)..(-1.3,-1.)
arc (-90:90:.2 and .6);
\draw [dashed](-2.8,.9) ..controls (0,-.5)..(2.8,.9);
\draw [dashed](-2.8,-1.7) ..controls (0,-.4)..(2.8,-1.7);
\draw [ dashed](-2.8,-1.7) arc (-70:70:.5 and 1.4);
\draw [ dashed](-2.8,.9) arc (110:250:.4 and 1.4);
\draw [ dashed](2.8,-1.7) arc (-70:70:.5 and 1.4);
\draw [  dashed](2.8,.9) arc (110:250:.4 and 1.4);
\draw (0,-.75) arc (-90:90:.1 and .3);
\draw [dashed](0,-.15) arc (90:270:.1 and .3);
\draw [dashed](-1.3,.2) arc (90:270:.2 and .6);
\draw[->, very thick] (-1,-.5) -- (-.3,-.5);
\draw[->, very thick] (1.4,-.5) -- (.7,-.5);
\end{tikzpicture}

\end{tabular}

\caption{A sketch of Definition \ref{DefN1}. The top part shows  tubular neighborhood $N_2$ along with the corresponding loops. In the ideal case we would have $N_1=N(\alpha, D_1/2)$. The bottom parts demonstrate required deformation contractions of (5).}
\label{FigDef}
\end{figure}
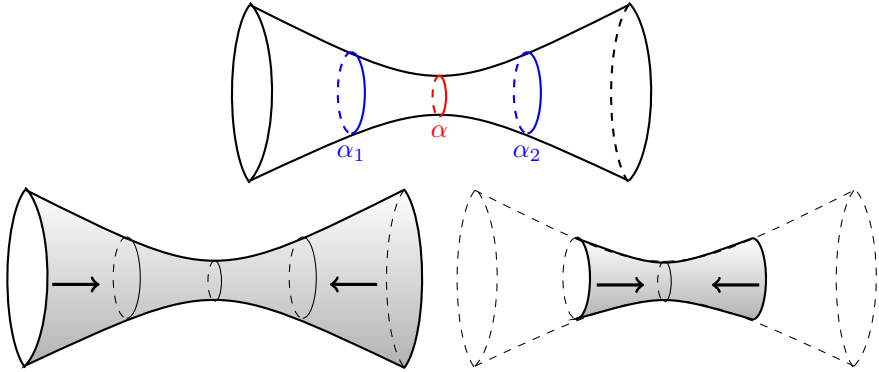

\begin{Rem}
 The conditions of Definition \ref{DefN1} stipulate that $\alpha$ has a sufficiently tame neighborhood, in which it is the shortest homotopy representative of its class. 
 Condition (5) implies $N_2 \DC \alpha$. Sufficient conditions implying condition (4),  i.e., conditions for $r>0$ and a topological circle $S$ in a metric space to have $\Rips(S,r)\simeq S^1$, are provided in \cite{ZV2}. For example, (4) holds if for each $i$, the cover of $\alpha_i$ by maximal open sets of diameter $r$ is a good cover. If the $\alpha_i$ are isometric to planar loops \cite{ACW} also provides sufficient conditions.
\end{Rem}

\begin{Lem}
 \label{Lem2021Add1}
Using the notation of Definition \ref{DefN1} we have   $N_2  \DC N_1 $.
\end{Lem}

\begin{proof}
Let  $F \colon N_2 \setminus \Int(N_1) \DC \di N_1$. It suffices to show that for each pair of points $x\in N_2 \setminus N_1, y\in N_1$ we have
  $d(F(x,t'),F(y,t')) \leq d(F(x,t),F(y,t))$  for all $ t'>t$. Choose a geodesic $g$ from $x$ to $y$ and let $y^*$ be the intersection of $g$ and $\di N_1$. By the assumption $F$ is mapping $x$ ever closer to $y^*$ and thus also to $y$.
\end{proof}

\begin{Prop}
\label{PropRips1}
 Suppose $D >0$ and the loop $\alpha$ is $DC(D)$ isolated. Then for each $r<D$:
\begin{enumerate}
 \item $\Rips(N_2,r) \simeq \Rips(N_1,r) \simeq \Rips(\alpha, r)$, 
 \item $\Rips(N_2 \setminus \Int (N_1)) \simeq \Rips(\di  N_1)$, and 
 \item $\Rips(\di  N_1) =\Rips(\alpha_1, r) \sqcup \Rips(\alpha_2, r)$.
\end{enumerate}
\end{Prop}

\begin{proof}
(1) and (2) follows by definitions, Lemma \ref{Lem2021Add1}, and Proposition \ref{PropDC}. (3) follows by Definition \ref{DefN1}(3).
\end{proof}

We next provide our basic theorem about the footprint detection.
Since the one-dimensional persistence was completely classified in \cite{ZV}, we focus on higher-dimensional persistence. 

\begin{Thm}[Footprint detection for loops on surfaces]
 \label{ThmMain1}
Suppose $X$ is a geodesic surface, $\alpha$ is a $DC(D)$ isolated geodesic circle for some $D > 0$, and $G$ is a group.  Then 
 $\{H_k(\Rips(\alpha,r),G)\}_{r \leq D}$ is a direct summand of  $\{H_k(\Rips(X,r),G)\}_{r \leq D}$ via the inclusion induced map for all $k \geq 2$. 
 \end{Thm}

\begin{proof}
 We set  the Mayer-Vietoris long exact sequence using the notation of Definition \ref{DefN1} and Proposition \ref{PropRips1}. For a fixed $r\leq D$ define 
 $$
 A=\Rips(N_2,r)\simeq \Rips(\alpha, r),
 $$ 
 $$
 B= \Rips(X \setminus \Int (N_1),r),
 $$
 $$
 A\cap B= \Rips(N_2 \setminus \Int (N_1),r) \simeq \Rips(\alpha_1, r) \sqcup \Rips(\alpha_2, r) \simeq S^1 \sqcup S^1
 $$
 Since $H_k(A \cap B)=H_{k-1}(A \cap B)=0, \forall k\geq 3$, we extract the following exact sequences:
 $$
0= H_{k}(A \cap B)\to  H_{k}(A) \oplus H_{k}(B)
 \stackrel{f_k}{\to} H_{k}(\Rips(X,r)) \to H_{k-1}(A \cap B)=0.
 $$
 This proves that $H_{k}(A)\cong H_k(\Rips(\alpha, r))$ is a direct summand in $H_k(\Rips(X,r)), \forall k\geq 3$. By Theorem \ref{ThmAA} $H_2(\Rips(\alpha, r))$ is trivial, hence the conclusion holds for all $k\geq 2$ and fixed $r$. 
 
 When considering a range $r\in (0,D)$, note that the bonding maps on persistences are induced by the inclusions, just as maps $f_k$ in the Mayer-Vietoris sequence. Hence all maps in question commute, implying that $\{H_k(\Rips(\alpha,r),G)\}_{r \leq D}$ is a direct summand of  $\{H_k(\Rips(X,r),G)\}_{r \leq D}$ for all $k \geq 2$, which completes the proof.
\end{proof}

Since $\alpha$ is a geodesic circle, its Rips filtration is known by Theorem \ref{ThmAA}, and so are the direct summands mentioned in Theorem \ref{ThmMain1}. Corollary \ref{CorDetect} summarizes such a  situation in terms of bars of persistence diagrams.

\begin{Cor}
\label{CorDetect}
 Suppose $X$ is a totally bounded geodesic surface, $\alpha$ is a   $DC(D)$ isolated geodesic circle for some $D>0$, and $G$ is a field. 
\begin{enumerate}
 \item If $\frac{l}{2l +1} |\alpha| < D \leq \frac{l+1}{2l+3} |\alpha|$ for some $l\in \NN$, then the following conclusion holds: the $PD$ of $X$ contains for each $n\in \{1, 2, \ldots, l-1\}$ a $(2n+1)$-dimensional bar $\big(\frac{n}{2n +1} |\alpha|,\frac{n+1}{2n+3} |\alpha|\big]$ and a $(2l+1)$-dimensional bar $\big(\frac{l}{2l +1} |\alpha|, w \big]$, for some $w \in [D, \frac{l+1}{2l+3} |\alpha|]$, all generated by the included Rips complex of $\alpha$.
\item If $\frac{l}{2l +1} |\alpha| < D \leq \frac{l+1}{2l+3} |\alpha|$ and $\alpha$ is a member of some lexicographically shortest homology basis,  then the conclusion of (1) holds for all $n\in \{0,1, \ldots, l-1\}$.
\item If $D \geq |\alpha|/2$, then the conclusion holds for all $n\in \{1,2, \ldots\}$. 
\end{enumerate}
 \end{Cor}

\begin{proof}
 Since $X$ is totally bounded the induced persistent homology is q-tame hence the PD exists in each dimension. The validity of (1)  follows from Theorems \ref{ThmMain1} and \ref{ThmAA}. Statement (2) follows from \cite{ZV}, and (3)  follows  from Theorems \ref{ThmMain1} and \ref{ThmAA}.
\end{proof}


Theorem \ref{ThmMain1} considers a case of a single loop $\alpha$. Using an inductive argument  Corollary \ref{CorDetect456} demonstrates that sufficiently disjoint geodesic circles generate separate footprints. In particular, for each loop of such a collection of geodesic circles we get distinct bars mentioned in Corollary \ref{CorDetect}.

\begin{Cor}
 \label{CorDetect456} 
 Suppose $X$ is a geodesic surface, $m\in \NN$, $\alpha_1, \ldots, \alpha_m$ are $DC(D)$ isolated geodesic circles for some $D > 0$ with the corresponding $N_2$ neighborhoods $N_{2,1}, \ldots, N_{2,m}$ from Definition \ref{DefN1} being disjoint, and $G$ is a group.  Then 
 $\oplus_{i=1}^m\{H_k(\Rips(\alpha_i,r),G)\}_{r \leq D}$ is a direct summand of  $\{H_k(\Rips(X,r),G)\}_{r \leq D}$ via the inclusion induced map for all $k \geq 2$. 
\end{Cor}

\begin{proof}
Assume $N_{1,1}, \ldots, N_{1,m}$ are neighborhoods of $\alpha_1, \ldots, \alpha_m$ corresponding to neighborhood $N_1$ from Definition \ref{DefN1}. 
Inductively apply the proof of Theorem \ref{ThmMain1} to $X, \quad X \setminus \Int(N_{1,1}), \quad X \setminus  ( \Int(N_{1,1}) \cup \Int(N_{1,2}))$, ...
\end{proof}

Theorem \ref{ThmMain1} is a prototype for footprint detection of subspaces onto which there are nice deformation contractions, and can be adapted to other subspaces. It is essentially a Hopf-type effect in persistence: one-dimensional geometric features generate higher-dimensional algebraic objects. It explains the existence of all blue bars and the red $3$-dimensional bar in Figure \ref{FigEssence}.

\section{A general framework for footprint detection via deformation contractions}
\label{SectSec}

Using the idea of the previous section we provide a general framework for footprint detection in metric spaces via deformation contractions. In this case we are detecting parts of the persistent homology of a subset $Z\subset X$ in the persistent homology of $X$. 

Suppose $X$ is a (not necessarily geodesic) metric space and let $Z\subset X$ be a  subspace. When considering a subset of $X$ as a metric space, we always assume it is equipped with the restriction of the metric on $X$.

\begin{Def}
 \label{DefN2}
Suppose $0 < a \leq b$, $G$ is a group and $k\in \NN$. A subspace $Z$ of a metric space $X$ is $DC(\langle a, b \rangle ; k, G)$ isolated if there exist two closed nested  neighborhoods $N_1 \subset N_2$ of $Z$, so that:
\begin{enumerate}
 \item $N_2 \supset N(N_1, r),  \forall  r\in \langle a, b \rangle$;
 \item  for each $r\in \langle a, b \rangle$, the condition $H_k(\Rips(Z,r),G)\neq 0$ implies that the following maps are trivial:
	\begin{enumerate}
	\item the inclusion-induced maps $H_k(\Rips(\di N_1,r),G) 
	\to H_k(\Rips( N_1,r),G)$ and $H_k(\Rips(\di N_1,r),G) \to 
	H_k(\Rips( X \setminus \Int(N_1),r),G)$;
	\item the boundary map 
	$$
	H_k(\Rips(X,r),G) \to H_{k-1}(\Rips(N_2 \setminus \Int 	
	(N_1),r))
	$$
	arising from the Mayer-Vietoris long exact sequence for a 	
	decomposition of $\Rips(X,r)$ into $A=\Rips(N_2,r)$ and 
	$B= 
	\Rips(X \setminus \Int (N_1),r)$;
	\end{enumerate}
 \item $N_2 \setminus \Int(N_1) \DC \di N_1, N_2 \DC N_1, $ and $N_1 \DC Z$.
\end{enumerate}
Notation $SDC(\langle a, b \rangle ; k, G)$ denotes the strict version of the defined property, i.e., a version in which all deformation contractions of (3) are strict deformation contractions.
\end{Def}

Proposition \ref{PropRips15} is an adaptation of Proposition \ref{PropRips1} and can be proved in the same way. 

\begin{Prop}
\label{PropRips15}
 Suppose $0<a\leq b$, $X$ is a metric space,  and $Z\subset X$ is a subspace for which (1) and (3) of Definition \ref{DefN2} hold. Then for each $r\in \langle a, b \rangle$:
\begin{enumerate}
 \item $\Rips(N_2,r) \simeq \Rips(N_1,r) \simeq \Rips(Z, r)$, and
 \item $\Rips(N_2 \setminus \Int (N_1)) \simeq \Rips(\di  N_1)$.
\end{enumerate}
\end{Prop}

\begin{Thm}[Footprint Detection Framework]
 \label{ThmMain15}
Suppose $X$ is a metric space, $G$ is a group, $ k\in \NN, 0<a\leq b$, and $Z\subset X$ is $DC(\langle a, b \rangle; k, G)$ isolated.  Then 
 $\{H_k(\Rips(Z,r),G)\}_{r\in \langle a, b \rangle}$ is a direct summand of  $\{H_k(\Rips(X,r),G)\}_{r\in \langle a, b \rangle}$ via the inclusion induced map.
 \end{Thm}

\begin{proof}
 The proof follows the same structure as that of Theorem \ref{ThmMain1}, using Proposition \ref{PropRips15} and the conditions of Definition \ref{DefN2} in the corresponding Mayer-Vietoris sequence for a decomposition of $\Rips(X,r)$ into $A=\Rips(N_2,r)$ and $B= \Rips(X \setminus \Int (N_1),r)$.
\end{proof}

In our future work we intend to tackle the problems of obtaining persistences of simple spaces (such as spheres) and controlling the homology of a neighborhood. With such results we could provide more convenient conditions for specific footprint detection situations introduced here.

\section{Nullhomologies of loops}
\label{SectNull}

In this technical section we provide nice nullhomologies of samples of loops, which will be used in Section \ref{Sect2D}.
%
Definition \ref{DefCirc} introduces a cyclic order of points on a loop. 

\begin{Def}\label{DefCirc}
 Choose an orientation on $S^1$. Notation $x_0 \prec x_1 \prec \ldots \prec x_k \prec x_0$ means that points $x_i$ appear on $S^1$ in the suggested order along the chosen orientation of $S^1$. In particular, this means that for each $i \in \NN \mod (k+1)$, the interval $(x_i, x_{i+1})$ on $S^1$  along the chosen orientation contains no point $x_j$. Equivalently, the collection of the mentioned intervals $(x_i, x_{i+1})$ on $S^1$  is disjoint.

If $\alpha \colon S^1 \to X$ is a loop, then $\alpha(t_0) \prec \alpha(t_1) \prec \ldots \prec \alpha(t_k) \prec \alpha(t_0)$ means that $t_0 \prec t_1 \prec \ldots \prec t_k \prec t_0$ along the chosen orientation.
\end{Def}

Throughout this section we will be referring to a loop $(A,d)$ as a (not necessarily geodesic) metric space homeomorphic to $S^1$. In this context we will also be using the derived metric space $(A', d')$, which denotes set $A$ equipped with the geodesic metric, i.e., a metric where the distance between two points is the length of the shortest segment in $(A,d)$ between the points. 
As a technical prerequisite we will assume that the circumference of $(A,d)$ is finite, so that the distance $d'$ attains only finite values.

\begin{Def}
 Suppose $A$ is a loop. Three points $x_0, x_1, x_2\in A$ are equidistant on $A$, if for some chosen orientation on $A$, $x_0 \prec x_1 \prec x_2 \prec x_0$ and the lengths of the closed intervals $[x_{0},x_{1}], [x_{1},x_{2}]$ and $[x_{2},x_{0}]$ along $A$ are the same.
\end{Def}

It is clear that the definition of equidistant points is independent of the choice of an orientation. The following lemma shows that three equidistant points, or even their appropriate approximations, may be used to obtain nice nullhomotopies of loops in Rips complexes.

\begin{Lem}
\label{LemPre1}
 Suppose $(A,d)$ is an oriented loop of circumference $1$ equipped with a (not necessary geodesic) metric $d$, and choose $r> 1/3$. Let $L$ be an $r$-sample of $A$ given by $(x_0, x_1, \ldots, x_k, x_{k+1}=x_0)$. 
Suppose there exist $t_0<t_1<t_2$, so that $x_{t_0} \prec x_{t_1} \prec x_{t_2} \prec x_{t_0}$ with the lengths of segments of the closed intervals $[x_{t_0},x_{t_1}], [x_{t_1},x_{t_2}]$ and $[x_{t_2},x_{t_0}]$ along $A$ being less than $r$.

Then there exist pairwise different triangles $\sigma_i$ in $\Rips(L, r)$, so that:
\begin{itemize}
\item $L=\di \sum_{i=1}^k \sigma_i$, and
\item if $x_p\in L$ and $x_q\in L$ are contained in some $\sigma_i$, then the length of the shortest segment on $A$ between $x_p$ and $x_q$ is less than r.
\end{itemize}
Furthermore, if $L=\di \sum_{j=1}^m \sigma'_j$ is another such decomposition, then $\sum_{i=1}^k \sigma_i-\sum_{j=1}^m \sigma'_j$ is a boundary in $\Rips(A,r)$.
\end{Lem}

\begin{proof} 
The triangles $\sigma_i$ are depicted in Figure \ref{FigNullhomotopy}. The shaded triangle is $(x_{t_0}, x_{t_1}, x_{t_2})$. For all $l\in \ZZ (\bmod \ 3)$ and for all $t_l < p< t_{l+1 \! (\bmod 3)}$ add a cone (with the appropriate orientation) over the segment $(x_{p-1}, x_{p})$ with apex $x_{t_l}$. Hence we have satisfied the two bullet points.

The construction works even if two of the points $x_{t_l}$ coincide: in such a case we remove those triangles $\sigma_i$ in the expression above which become degenerate. 

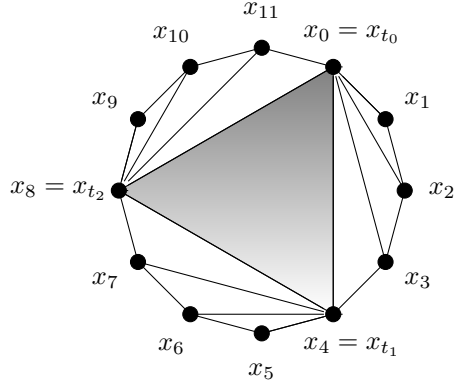
\begin{figure}
\begin{tikzpicture}[scale=1]

\node (q1) at (60:1.9){}; 

\node (q2) at (180: 1.9){};

\node (q3) at (-60:1.9){};

\draw[top color=gray,bottom color=white, fill opacity =1] (60:1.9) -- (180:1.9) -- (-60:1.9) -- cycle;
\foreach \x  in {1,2,3,5,6,7,9,10,11}
	\draw (-\x * 30+90:1.9) -- (-\x * 30+90:2.0) (-\x * 30+60:2.4) node {$x_{\x }$};
\foreach \x  in {0}
	\draw (-\x * 30+90:1.9) -- (-\x * 30+90:2.0) (-\x * 30+60:2.4) node {$x_{\x }=x_{t_0}$};
\foreach \x  in {4}
	\draw (-\x * 30+90:1.9) -- (-\x * 30+90:2.0) (-\x * 30+60:2.4) node {$x_{\x }=x_{t_1}$};
\foreach \x  in {8}
	\draw (-\x * 30+90:1.9) -- (-\x * 30+90:2.0) (-\x * 30+60:2.7) node {$x_{\x }=x_{t_2}$};
\foreach \x  in {0, 1, ..., 11}
	\draw [fill=black] (-\x * 30+90:1.9) circle (.1cm);
\foreach \x  in {0, 1, ..., 11}
	\draw (-\x * 30+90:1.9) -- (-\x * 30+120:1.9);


\foreach \x  in { 2, 3,..., 6}
	\draw (q1) -- (-\x * 30+120:1.9);
\foreach \x  in { 6,7, ..., 10}
	\draw (q3) -- (-\x * 30+120:1.9);
\foreach \x  in {10, 11, 12, 1, 2}
	\draw (q2) -- (-\x * 30+120:1.9);

\end{tikzpicture}
\caption{Sketch of proof of Lemma \ref{LemPre1}.}
\label{FigNullhomotopy}
\end{figure}
 
To prove the second part let $(A', d')$ denote set $A$ equipped with the geodesic metric, i.e., a metric where the distance between two points is the length of the shortest segment in $(A,d)$ between the points.  
 Each $\sigma_i$ mentioned above is also contained in $\Rips(A', r)$, as only lengths of the connecting path segments were used in the argument and no special properties of the metric were assumed. 
 The second homology group of $\Rips(A', r)$ is trivial for all $r$ and all coefficients by Theorem \ref{ThmAA}, hence $\sum_{i=1}^k \sigma_i-\sum_{j=1}^m \sigma_j$ is a boundary in $\Rips(A',r)$. Since the identity $(A',d') \to (A,d)$ is a contraction, the induced map $\Rips(A',r)\to \Rips(A,r)$ is an inclusion, hence $\sum_{i=1}^k \sigma_i-\sum_{j=1}^m \sigma_j$ is a boundary in $\Rips(A,r)$ as well.
\end{proof}

\begin{Cor}
\label{CorPre1}
 Suppose $(A,d)$ is a loop of circumference $1$ equipped with a (not necessary geodesic) metric $d$, and choose $r> 1/3$. Let $L$ be a $(r-1/3)$-sample of $A$ given by $(x_0, x_1, \ldots, x_k, x_{k+1}=x_0)$. 
Then there exist pairwise different triangles $\sigma_i$ in $\Rips(L, r)$, so that $L=\di \sum_{i=1}^k \sigma_i$.
Furthermore, if $L=\di \sum_{j=1}^m \sigma'_j$ is another such decomposition, then $\sum_{i=1}^k \sigma_i-\sum_{j=1}^m \sigma'_j$ is a boundary in $\Rips(A,r)$.
\end{Cor}

\begin{proof}
 Fixing an orientation of $A$, choose $t_0<t_1<t_2$, so that $x_{t_0} \prec x_{t_1} \prec x_{t_2} \prec x_{t_0}$ with the lengths of the segments of the closed intervals $[x_{t_0},x_{t_1}], [x_{t_1},x_{t_2}]$ and $[x_{t_2},x_{t_0}]$ along $A$ being less than $r$. For example, we can set $t_0=0$, choose $t_1$ so that the length of $A$ along the chosen orientation to $x_{t_1}$ is between $1/3$ and $r$ (such $x_{t_1}$ exists since $L$ is a $(r-1/3)$-sample), etc. 
\end{proof}

\section{Two-dimensional footprint}
\label{Sect2D}

In this section we prove the most technical result of this paper, describing an appearance of two-dimensional footprints. A sufficiently DC isolated geodesic circle in a compact geodesic space which is also a member of a shortest homology basis, generates a one-dimensional footprint by \cite{ZV}. If such a circle is not a member of a minimal homology basis, then the result of this section shows that it in fact induces a two-dimensional footprint. In contrast to the results of the previous sections, this footprint does not appear in the persistence of the circle itself but arises from the geometry of the entire space. As described below, such a footprint can be used to separate geodesic circles, detectable by higher dimensional footprints as described in the previous sections, into two groups: members of a shortest homology basis, and the rest of the geodesic circles. As an added benefit, the mentioned two-dimensional footprint can provide a good approximation for homotopy height.

The results in this section focus on geodesic circles on surfaces. We intend to describe more general results in a forthcoming paper.

\begin{Thm}
 \label{ThmMain2}
Suppose $X$ is a geodesic surface, $\alpha$ is an $DC(D,3 D/2)$ isolated loop for some $D > |\alpha|/3$, and $G$ is a group.  Assume $\alpha$ is homologous in $H_1(X,G)$ to a $G$-combination of loops $\beta_1, \beta_2, \ldots, \beta_k$ of length at most $|\alpha|$, none of which intersects $N_1$. Then the following hold.
\begin{enumerate}
\item For each $r\in (|\alpha|/3, D]$ there exists a non-trivial $Q_r\in H_2(\Rips(X,r),G)$ so that:
\begin{enumerate}
\item  For each pair $q_1< q_2$ of parameters from $(|\alpha|/3, D]$ we have $i^G_{q_1,q_2}(Q_{q_1})=Q_{q_2}$.
\item For any $q\in (|a|/3, D]$ there exists no $q_0\leq |\alpha|/3$, for which $Q_q$ is in the image of $i^G_{q_0, q}$.
 \item If  $\alpha$ is homotopic to some shorter geodesic circle $\beta$  in $X$ and $3 q_3$ is larger than the homotopy height between $\alpha$ and $\beta$, then $i^G_{q,q_3}(Q_{q})$ is trivial for any $q\in (|a|/3, \min(D,q_3)]$.
\end{enumerate}
  
  \item If $G$ is a field and $\{H_2(\Rips(X,r),G)\}_{r>0}$ is $q$-tame, then the persistence $\{H_2(\Rips(X,r),G)\}_{r<D}$ contains $G_{(|\alpha|/3, w'/3 )}$ as a direct summand for some $w\in (|\alpha|/3, \min(D,q_3))$. 
\end{enumerate}
\end{Thm}

\begin{proof}
\textbf{Proof of (1)}.
We will be using the singular homology representation in $X$: there exist singular $2$-simplices $\tilde\Delta_{\tilde j}$ and $\tilde h_{\tilde j}, g_i\in G$ in $X$ so that 
$$
[\alpha]_G = \sum_{i=1}^k g_i [\beta_i]_G + \di \sum_{\tilde j=1}^{\tilde k'} \tilde h_{\tilde j} [\tilde\Delta_{\tilde j}]_G.
$$

We now subdivide  singular $2$-simplices $\tilde\Delta_{\tilde j}$ into $\Delta_j$ so that for some $h_j\in G$
 $$
L_\alpha = \sum_{i=1}^k g_i L_i + \di \sum_{j=1}^{k'} h_j \Delta_j
$$
holds with the following conditions:
\begin{itemize}
 \item the diameter of each singular simplex $\Delta_j$ is less than $|\alpha|/3$;
 \item $L_\alpha$ and $L_i$ are subdivided loops $\alpha$ and $\beta_i$ correspondingly, with their vertices forming $(|\alpha|/3)$-samples of $\alpha$ and $\beta_i$ correspondingly;
 \item each $(|\alpha|/3)$-sample above contains three equidistant points on the corresponding loop (this will allow us to apply Lemma \ref{LemPre1} below).
\end{itemize}

Using the condition on the diameter of simplices we may abuse the notation and consider $L_\alpha$ and $L_i$ to be either subdivided singular loops in $X$ (in which case they correspond to $\alpha$  and $\beta_i$) or, by retaining the vertices, $(|\alpha|/3)$-loops in $\Rips(X, |\alpha|/3)$. The same goes for each $\Delta_j$: its three vertices form a simplex in $\Rips(X, |\alpha|/3)$, which we also denote by $\Delta_j$. 

Fix $r\in (|\alpha|/3, D]$. Using Lemma \ref{LemPre1} we can express each $L_i$ as  a chain $L_i=\di \sum_{p=1}^{k_p} \tau_{i,p}$, where each $\tau_{i,p}$ is a $2$-simplex in $\Rips(\beta_i, r)$. Hence we obtain an expression of a chain in $\Rips(X, r)$: 
$$
L_\alpha = \di \sum_{i=1}^k  \sum_{p=1}^{k_p} g_{i}\tau_{i,p} + \di \sum_{j=1}^{k'} h_j \Delta_j
$$

 We use Lemma \ref{LemPre1}  once more to obtain an expression of a chain in $\Rips(\alpha, r): L_\alpha = \di \sum_{l=1}^{k_\alpha}  \sigma_l$, implying
\begin{equation}
\label{Eq7.1}
\di \sum_{l=1}^{k_\alpha}\sigma_l= \di \sum_{i=1}^k \sum_{p=1}^{k_p} g_{i}\tau_{i,p} + \di \sum_{j=1}^{k'} h_j \Delta_j.
\end{equation}
See Figure \ref{Fig2DFootprint} for a demonstration of such an expression.

\begin{figure}[htbp]
\begin{center}
\includegraphics[scale=.6]{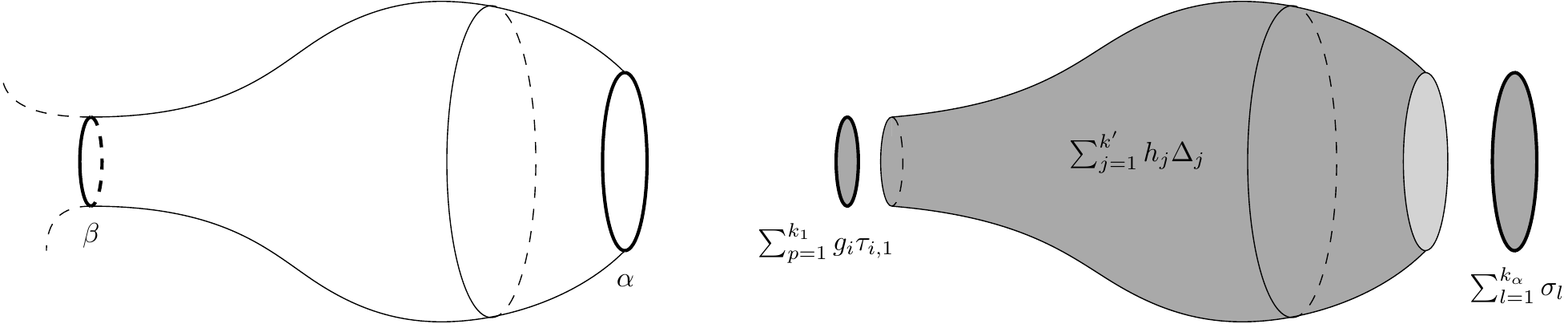}
\caption{An explanation of Equation \ref{Eq7.1}. The left side of the figure represents a portion of a surface containing a geodesic circle $\alpha$ and a shorter homotopic geodesic circle $\beta$. The right side demonstrates the geometric objects, whose triangulations are parts of Equation \ref{Eq7.1}. Simplices $\Delta_j$ triangulate the portion of the surface between the loops while simplices $\sigma_l$ and $\tau_{i,1}$ triangulate the ``lids'' arising from nullhomotopies of $\alpha$ and $\beta$ in the Rips complex (as described in Lemma \ref{LemPre1}), turning the total expression into a $2$-cycle.}
\label{Fig2DFootprint}
\end{center}
\end{figure}

This means that 
$$
C_r=\sum_{l=1}^{k_\alpha} \sigma_l - \sum_{i=1}^k  \sum_{p=1}^{k_p} g_{i}\tau_{i,p} -  \sum_{j=1}^{k'} h_j \Delta_j
$$
is a cycle and we define $Q_r$ to be the homology class represented by $C_r$. By standard arguments we see that a subdivision of singular $2$-simplices $\Delta_j$ above results in a different (finer) representation $C_r$, but does not change $Q_r$ (a subdivision replacing each $\Delta_j$ by the corresponding homologous sum of contained triangles would induce finer samples $L_\alpha$ and $L_i$ in the expression of $ \di \sum_{j=1}^{k'} h_j \Delta_j$).  These finer samples are equivalent to the original choice by \cite[Proposition 3.2(4)]{ZV}, and so are their nullhomologies by Lemma \ref{LemPre1}.

We next prove that $Q_r\in H_2(\Rips(X,r),G)$ is non-trivial. 
Consider the following excerpt of the Mayer-Vietoris sequence set up in Theorem \ref{ThmMain1}:
\begin{align}
\label{Eq01}
H_2(\Rips(X,r))\to H_1(A\cap B). 
\end{align}
We will be making use of the notation set up by the proof of Theorem \ref{ThmMain1} and Definition \ref{DefN1}. 
In order to prove $Q_r$ is non-trivial, we will show that its image via the boundary map of Equation \ref{Eq01} is non-trivial.

We first decompose $C_r$ into two parts $C_r = L_A + L_B$ as follows. $L_A$ consists of  $\sum_{l=1}^{k_\alpha} \sigma_l$ plus $\widetilde L_A$.  $\widetilde L_A$ consists of all the summands $-h_j \Delta_j$, for which the singular simplex in $X$ corresponding to $\Delta_j$  intersects $N_1$. $L_B$ consists of all other summands of  $C_r$. Observe that $L_A$ is contained in $A$, and $L_B$ is contained in $B$ as $D>r$. 
We will also be interested in the boundary of $L_A$. 
The boundary of $\sum_{l=1}^{k_\alpha}\sigma_l$ is $L_\alpha$. The boundary of $\widetilde L_A$ is null-homologous in $A$ (recall it arises from $\alpha$ being homologous to a combination of shorter loops in $X$ and each involved singular $2$-simplex $\Delta_j$ is contractible in the tube $N_2$ as it is of diameter less than $|\alpha|/3$, hence so is its counterpart in $A$ by  \cite[Proposition 3.2(7)]{ZV}). Putting them together we see that the boundary of $L_A$ has winding number $\pm 1$ (the sign being dependent on the chosen orientation) in $A$. This boundary represents the sum of the components (in $H_1(A \cap B)=G \oplus G$) of the image of $Q_r$ via the boundary map of Equation \ref{Eq01}, hence $Q_r$ is non-trivial.

\textbf{Proof of statement (a)}.

The statement holds as:
\begin{itemize}
\item   parts $\di \sum_{j=1}^{k'} h_j \Delta_j$ of $C_{q_1}$ and $C_{q_2}$ are the same;
\item the remaining parts (nullhomologies of loops $L_\alpha$ via Lemma \ref{LemPre1}) can also be chosen to be the same, and are in general homologous by Lemma \ref{LemPre1}.
\end{itemize}

\textbf{Proof of statement (b)}.

Consider the following diagram:
\begin{align}
\label{Dia2}
\xymatrix{
H_2(\Rips(X,q)) \ar[r] & H_1(\Rips(N_2,q)\cap \Rips(X \setminus \Int(N_1),q)) \\
H_2(\Rips(X,q_0)) \ar[r]^-h \ar[u]
& H_1(\Rips(N_2,q_0)\cap \Rips(X \setminus \Int(N_1),q_0)),\ar[u]
}
\end{align}
where the horizontal maps are excerpts from the Mayer-Vietoris sequence set up in Theorem \ref{ThmMain1} and the vertical maps are inclusion induced. It is easy to verify that the diagram commutes. As was mentioned before, $\Rips(N_2,q)\cap \Rips(X \setminus \Int(N_1),q)$ and $\Rips(N_2,q_0)\cap \Rips(X \setminus \Int(N_1),q_0)$ are both homotopy equivalent to a disjoint union of two copies of $S^1$. The difference we are about to utilise is that while $q_0$-samples
 of $\alpha_i$ are not contractible in $\Rips(N_2,q_0)$, they are contractible in $\Rips(N_2,q)$. Hence, for example, a $q_0$-sample of $\alpha_1$ may appear as a boundary of a $2$-chain in $\Rips(N_2,q)$, a fact we used earlier in the proof, but not in $\Rips(N_2,q_0)$, which we will use here. 
 
Suppose $K_{q_0}$ is a $2$-chain in $\Rips(X,q_0)$. Keeping in mind the diagram above, we decompose $K_{q_0}=K_{q_0}^1 + K_{q_0}^2$, with $K_{q_0}^1$ being a chain in $\Rips(N_2, q_0)$ and $K_{q_0}^2$ being a chain in $\Rips(X \setminus \Int(N_1),q_0)$, in the following way.
 Let $K_{q_0}^1$ consist of all summands $g \Delta$ of $K_{q_0}$ with $g\in G, \Delta \in \Rips(X, q_0)$, for which at least one vertex of $\Delta$ is contained in $N_1.$ Take the three vertices of such $\Delta$ and connect them pairwise by geodesics to obtain a triangle $S$. Since $3D/2 > 3 q_0 /2$ and at least one vertex is in $N_1$, the obtained triangle is contained in $N_2$. Furthermore, as the circumference of $S$ is less than $|\alpha|$ and $N_2 \DC \alpha$, we conclude that $S$ is contractible in $N_2$. This easily implies (for an argument see for example \cite[Proposition 3.2, (7)]{ZV}) that $\di \Delta$ is contractible in $\Rips(N_2, q_0)$ hence the entire boundary of $K_{q_0}^1$ is nullhomologous in $H_1(\Rips(N_2, q_0))$. This boundary represents the sum of the components (in $H_1(A \cap B)=G \oplus G$) of the image of $[K_{q_0}]_G\in H_2(\Rips(X,q_0))$ via the lower horizontal map $h$ of the diagram (\ref{Dia2}) above. 
 While its image in 
 $$
 H_1(\Rips(N_2,q_0)\cap \Rips(X \setminus \Int(N_1),q_0)) \cong G \oplus G
 $$
 may be non-trivial, the two components of $h([K_{q_0}]_G)$ have, up to the sign depending on the chosen orientation, the same winding number, as they have to add up to $0$. This contrasts the property of $Q_r$ above, where we proved that the difference between the corresponding winding numbers is $\pm 1$. Hence $Q_r$ can't be an image of any element of $H_2(\Rips(X,q_0))$ via the natural inclusion induced map, which proves (b).
 
\textbf{ Proof of statement (c)}. 

Assume there is a homotopy $H \colon I \times I \to X$ (we will use the notation $H_t=H|_{\{t\}\times I}$) realizing the homotopy height $w$ between $\alpha$ and $\beta$:
\begin{itemize}
  \item $H_0=\alpha$;
  \item $H_1=\beta$;
  \item $H(t,1)=H(t,0), \forall t\in I$;
  \item $|H_t(I)| \leq w, \forall t\in I$.
  \end{itemize}
  If the homotopy height $w$ can't be realized precisely, we choose  $\tilde w \in (w, q_3)$ and use  a similar homotopy $H$ with the condition $|H_t(I)| \leq w  $ being replaced by $|H_t(I)| \leq \tilde w, \forall t\in I$.  In this case the argument below also provides the same conclusion \textbf{(c)}. For the sake of simplicity we may now assume that the homotopy height is realized precisely by $H$.

\begin{figure}[htbp]
\begin{center}
\includegraphics[scale=.7]{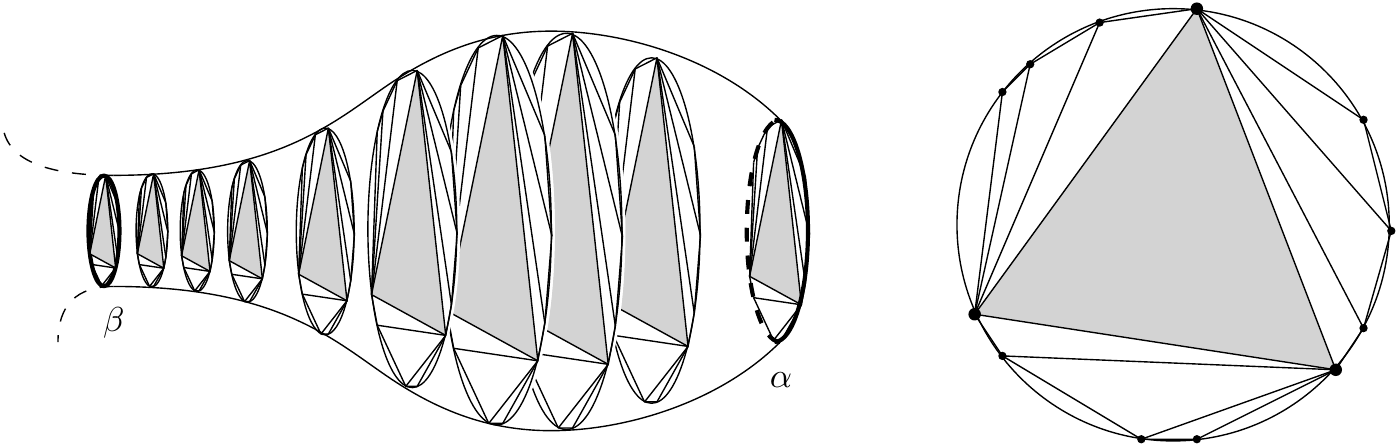}
\caption{Idea of proof of part (c) of Theorem \ref{ThmMain2}. The tubular cycle described by Figure \ref{Fig2DFootprint} (on the left) is filled by inserting parallel copies of nullhomotopies provided by Lemma \ref{LemPre1} (on the right).}
\label{Fig2DFootprint2}
\end{center}
\end{figure}
 
 Define $\eps=q_3 /3-w>0$. Subdivide $I \times I$ into a  grid of size $M \times M$ for some even $M\in \NN$, consisting of squares $S_{m,n}=[\frac{m}{M}, \frac{m+1}{M}]\times [\frac{n}{M}, \frac{n+1}{M}]$, so that the diameter of each $H(S_{m,n})$ is less than $\frac{\eps}{20}$. In order to obtain a proper triangulation $T$ of $I \times I$ add diagonals to squares $S_{m,n}$: if $m$ is odd add a diagonal from $(\frac{m}{M},\frac{n}{M})$ to $(\frac{m+1}{M}, \frac{n+1}{M})$; if $m$ is even add a diagonal from $(\frac{m}{M},\frac{n+1}{M})$ to $(\frac{m+1}{M}, \frac{n}{M})$, for all $n$ (see the right side of Figure \ref{FigNullhomotopy1} for an excerpt). This way we subdivide each square $S_{m,n}$ into two triangles: the upper triangle $T^u_{m,n}$ and the lower triangle $T^l_{m,n}$. By the standard argument about equivalence of singular and simplicial homology, we may now replace $\sum_{j=1}^{k'} h_j \Delta_j$ in the expression of $Q_r$ by $\sum_{m,n=1}^M (H(T^u_{m,n})+H(T^l_{m,n}))$. Furthermore, we replace $\sum_{i=1}^k  \sum_{p=1}^{k_p} g'_{i,p}\tau_{i,p}$ by $ \sum_{p=1}^{k_\beta} \tau_{p}$ according to Lemma \ref{LemPre1}, and thus obtain another representative of $Q_r$ in $\Rips(X,r)$: 
$$
L=\sum_{l=1}^{k_\alpha} \sigma_l  - \sum_{p=1}^{k_\beta} \tau_{p} - \sum_{m,n=1}^M (H(T^u_{m,n})+H(T^l_{m,n})).
$$
Note that again we abuse the notation in the sense that we think of $H(T^u_{m,n})$ as a triangle in $\Rips(X,r)$ determined by the images in $X$ of the three vertices of $T^u_{m,n}$  by $H$.

This essentially determines a map from the boundary of the closed cylinder to $\Rips(X,r)$, with homotopy $H$ representing the map on $I \times S^1$ and the nullhomotopies of $\alpha$ and $\beta$ representing the lids. The idea now is to use the fact that for each $t\in I$, the $\frac{\eps}{20}$-loop induced by $T$ on $H_t$ is contractible in $\Rips(X,q_3)$, and thus extend $H$ to the whole solid cylinder.

Fix  $m\in\{0,2,4,\ldots, M-2\}$. Set $t_0=t_3=0$. Let $L^m$ and $L^{m+2}$ denote $\eps/20$-loops obtained by restricting $H_m$ and $H_{m+2}$ respectively to the vertices of $T$. We first choose approximate equidistant points on $H_m(I)$ to facilitate the use of Lemma \ref{LemPre1}, as was done in the proof of Corollary \ref{CorPre1}. 
\begin{itemize}
 \item Let $t_1$ be the smallest integer in $\{1,2,3,\ldots, M\}$ so that $|H_{m}([0,t_1/M])|> w/3$; or $t_1=M$, if the condition is not satisfied.
 \item  Let $t_2$ be the smallest integer in $\{2,3,\ldots, M\}$ so that $|H_{m}([t_1/M,t_2/M])|> w/3$; or $t_2=M$, if the condition is not satisfied.
\end{itemize}
 Note that for each $i , |H_m([t_i/M,t_{i+1}/M])| < w/3 + \eps/20$. By the bound on the size of triangles in $T$ we also have $\diam(H_{m+2}([t_i/M,t_{i+1}/M]))< w/3 + \eps/20 + 2\eps/20=w/3 + 3 \eps/20< q_3/3$. We now apply Lemma \ref{LemPre1} for $L^m$ and $L^{m+2}$ using the decomposition into three intervals $H_{m}( [t_i, t_{i+1}]), i\in\{0,1,2\}$ and $H_{m+2}( [t_i, t_{i+1}]), i\in\{0,1,2\}$, to get a single triangulation $T_m$ of an $M$-gon as in Figure \ref{FigNullhomotopy1}, so that $H_m$ and $H_{m+2}$ applied to the vertices of $T_m$ represent a nullhomotopy of $L^m$ and $L^{m+2}$ in $\Rips(X,w)$. 
When using triangulations $T_m$ we will generally be considering  the corresponding triangles (i.e., maximal simplices) as a decomposition of an $M$-gon.

Consider a cylinder, obtained by identifying points $(x,0)\sim (x,1)$ in $[M, M+2]\times I$.  Let $T_m^\mathcal{L}$ denote a copy of $T_m$ attached along the domain of $H_m$ in such a cylinder, and let $T_m^\mathcal{R}$ denote a copy of $T_m$ attached along the domain of $H_{m+2}$. Form a triangulation of a thin cylinder
$$
U_m = U_m^\mathcal{L} \cup U_m^\mathcal{R},
$$
where 
$$
U_m^\mathcal{L} =  \{T^u_{m,n} \}_{n=1}^M \cup \{T^l_{m,n}\}_{n=1}^M \cup T_m^\mathcal{L}
$$
and
$$
U_m^\mathcal{R} = \{T^u_{m+1,n}\}_{n=1}^M \cup \{T^l_{m+1,n}\}_{n=1}^M \cup T_m^\mathcal{R}.
$$
Because of alternating diagonals in the construction of $T$, $U_m^\mathcal{L}$ and $U_m^\mathcal{R}$ are isomorphic triangulations of a disc glued together along their boundaries corresponding to $\{m+1\}\times I\subset I\times I$. Applying $H$ to the vertices of $U_m$ we thus obtain two maps, which are contiguous in $\Rips(X,q_3)$ as $T$ is fine enough. Thus $\sum_{\sigma\in U_m} H(\sigma)$ is homologically trivial in $\Rips(X,q_3)$. 

For $m\in\{2,4,\ldots, M-2\}$ define also a triangulation (simplicial complex) $V_m=T_m^\mathcal{L} \cup T_{m-2}^\mathcal{R}$. By Lemma \ref{LemPre1} the chain $\sum_{\sigma\in V_m}H(\sigma)$ is homologically trivial in $\Rips(X,q_3)$.

\begin{figure}
\begin{tikzpicture}[scale=1]

\node (q1) at (60:1.9){}; 

\node (q2) at (180: 1.9){};

\node (q3) at (-60:1.9){};

\draw[top color=gray,bottom color=white, fill opacity =1] (60:1.9) -- (180:1.9) -- (-60:1.9) -- cycle;
\foreach \x  in {1,2,3,5,6,7,9,10,11}
	\draw (-\x * 30+90:1.9) -- (-\x * 30+90:2.0) (-\x * 30+60:2.4) node {$x_{\x }$};
\foreach \x  in {0}
	\draw (-\x * 30+90:1.9) -- (-\x * 30+90:2.0) (-\x * 30+60:2.4) node {$x_{\x }=x_{t_0}$};
\foreach \x  in {4}
	\draw (-\x * 30+90:1.9) -- (-\x * 30+90:2.0) (-\x * 30+60:2.4) node {$x_{\x }=x_{t_1}$};
\foreach \x  in {8}
	\draw (-\x * 30+90:1.9) -- (-\x * 30+90:2.0) (-\x * 30+60:2.7) node {$x_{\x }=x_{t_2}$};
\foreach \x  in {0, 1, ..., 11}
	\draw [fill=black] (-\x * 30+90:1.9) circle (.1cm);
\foreach \x  in {0, 1, ..., 11}
	\draw (-\x * 30+90:1.9) -- (-\x * 30+120:1.9);


\foreach \x  in { 2, 3,..., 6}
	\draw (q1) -- (-\x * 30+120:1.9);
\foreach \x  in { 6,7, ..., 10}
	\draw (q3) -- (-\x * 30+120:1.9);
\foreach \x  in {10, 11, 12, 1, 2}
	\draw (q2) -- (-\x * 30+120:1.9);
\draw (4.5, -3) to (4.5, 3);
\foreach \x  in {0,1,2,...,11}
	\draw (4,\x/2 -3) to (5, \x/2 -3) to (5, \x/2 -3 + .5) to (4,\x/2 -3+ .5) to cycle;
\foreach \x  in {0,1,2,...,11}
	\draw (4,\x/2 -3) to (4.5, \x/2 -3 + .5) to (5, \x/2 -3);
\foreach \x  in {0,1, ..., 11}
	\draw [fill=black] (4,\x/2 -3) circle (.1cm) node[left] {$x_{\x}$};
\foreach \x  in {0,1, ...,11}
	\draw [fill=black] (5,\x/2 -3) circle (.1cm) node[right] {$x_{\x}$};
	\foreach \x  in {12}
	\draw [fill=black] (4,\x/2 -3) circle (.1cm) node[left] {$x_{0}$};
\foreach \x  in {12}
	\draw [fill=black] (5,\x/2 -3) circle (.1cm) node[right] {$x_{0}$};

\end{tikzpicture}
\caption{Sketch of proof of part (c) of Theorem \ref{ThmMain2}. Triangulation $U_m$ of a cylinder is obtained by taking a vetrical strip of triangles of $T$ (the right part) and attaching a copy of a triangulation $T_r$ of a disc (shown on left) on each slice, so that the labels of the vertices coincide.}
\label{FigNullhomotopy1}
\end{figure}
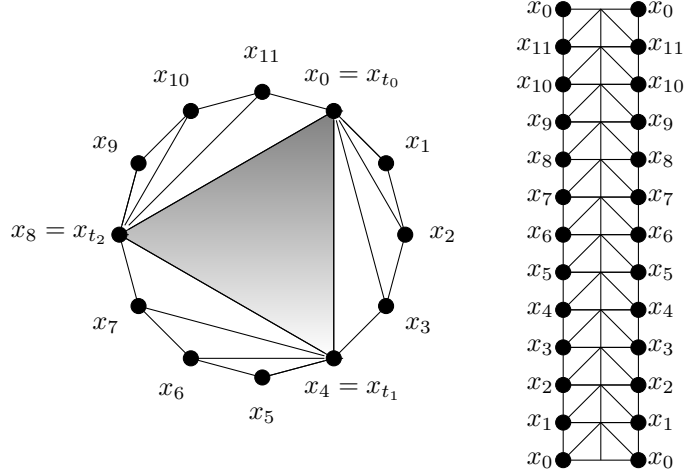

Observe now that the equality of chains:
$$
L=\sum_{m\in\{0,2,4,\ldots, M-2\}} \sum_{\sigma \in U_m} H(\sigma) - \sum_{m\in\{2,4,\ldots, M-2\}} \sum_{\sigma \in V_m} H(\sigma).
$$
Since each involved summand $\sum_{\sigma \in U_m} H(\sigma)$ and $\sum_{\sigma \in V_m} H(\sigma)$  is trivial in $\Rips(X,q_3)$, so is $L$.

 This concludes the proof of (1).
Statement \textbf{(2)} follows directly from (1).
\end{proof}

The following proposition explains the condition of Theorem \ref{ThmMain2} requiring loops being disjoint with $N_1$. It essentially excludes the case where $\alpha$ would have a small cylindric neighborhood, meaning that there would be loops of length $|\alpha|$ arbitrarily close to and homotopic to $\alpha$. Technically speaking, any such loop is an expression of $\alpha$ by a loop of length at most $\alpha$. However, in such a case the two-dimensional homology class constructed in Theorem \ref{ThmMain2} could be trivial. In order to avoid such a situation we require an empty intersection with $N_1$. The following proposition explains that this is indeed the case in a generic situation.

\begin{Prop}
 \label{PropCons}
Suppose $X$ is a geodesic surface, $\alpha$ is an $DC(D_1, D_2)$ isolated loop for some $D_1 > 0, D_2 \geq |\alpha|/2$, and $G$ is a group.  Assume $\alpha$ is homologous in $H_1(X,G)$ to a finite $G$-combination of  loops $\beta_i$ of lengths at most $|\alpha|$, none of which are equal to $\alpha$ (this holds, for example, if $\alpha$ is not a member of any lexicographically shortest homology basis). Then $\beta_i$ can be chosen to be geodesic circles. If all loops $\beta_i$ are shorter than $\alpha$ or $\alpha$ is $SDC(D_1, D_2)$ isolated, then  none of loops $\beta_i$ intersect $N_1$.
\end{Prop}

\begin{proof}
 The first part is contained in Theorem 8.4 in \cite{ZV}. 
 
If some $\beta_i$ intersected $N_1$, then by the condition on $D_2$ it would be contained in $N_2$ (see Definition \ref{DefN1} for notation), which is homeomorphic to a tube. The shortest generator of $H_1(N_2, G)$  is $\alpha$. If $|\beta_i|< |\alpha|$ then $\beta_i$ is contractible and obsolete.  If $\alpha$ is $SDC(D_1, D_2)$ isolated and $\beta_i \neq \alpha$, then the strict deformation retraction of $\beta_i$ is shorter than $\alpha$ and homotopic to $\beta_i$, thus contractible, and $\beta_i$ is again obsolete.
 \end{proof}

\section{Footprint detection using \v Cech complexes and closed filtrations}
\label{SectCech}

\subsection{\v Cech complexes}

An approach, equivalent to the one based on the Rips complexes in the previous sections, could also be developed for \v Cech complexes. Since the arguments are almost the same, we will only state the versions of required statements and the most general main results for the \v Cech complexes, and comment on the few modifications to the results.

\begin{Def}
 Suppose $A$ is a loop. Four points $x_0, x_1, x_2, x_3\in A$ are in a square-like formation on $A$, if for some chosen orientation on $A$, $x_0 \prec x_1 \prec x_2 \prec x_3 \prec x_0$ and the length segments of the closed intervals $[x_{0},x_{1}], [x_{1},x_{2}],[x_{2},x_{3}]$ and $[x_{3},x_{0}]$ along $A$ are the same.
\end{Def}

\begin{Lem}
\label{LemPre1C}
[A version of Lemma \ref{LemPre1}]
 Suppose $(A,d)$ is a loop of circumference $1$ equipped with a (not necessary geodesic) metric $d$, and choose $r> 1/4$. Let $L$ be an $r$-sample of $A$ given by $(x_0, x_1, \ldots, x_k, x_{k+1}=x_0)$. 
Suppose there exist $t_0<t_1<t_2<t_3$, so that $x_{t_0} \prec x_{t_1} \prec x_{t_2}\prec x_{t_3} \prec x_{t_0}$ with the length of segments of the closed intervals $[x_{t_0},x_{t_1}], [x_{t_1},x_{t_2}], [x_{t_2},x_{t_3}]$ and $[x_{t_3},x_{t_0}]$ along $A$ being less than $r$.

Then there exist pairwise different triangles $\sigma_i$ in $\C_L(L, r)$, so that:
\begin{itemize}
\item $L=\di \sum_{i=1}^k \sigma_i$, and
\item if $x_p\in L$ and $x_q\in L$ are contained in some $\sigma_i$, then the length of the shortest segment on $A$ between $x_p$ and $x_q$ is less than 1/4.
\end{itemize}
Furthermore, if $L=\di \sum_{j=1}^m \sigma'_j$ is another such decomposition, then $\sum_{i=1}^k \sigma_i-\sum_{j=1}^m \sigma_j$ is a boundary in $\C_A(A,r)$.
\end{Lem}

\begin{proof}
 The decomposition is given by Figure \ref{FigNullhomotopyC}.
\begin{figure}
\begin{tikzpicture}[scale=1]

\node (q1) at (90:1.9){}; 

\node (q2) at (180: 1.9){};

\node (q3) at (-90:1.9){};
\node (q4) at (0:1.9){};

\draw[top color=gray,bottom color=white, fill opacity =1] (90:1.9) -- (180:1.9) -- (-90:1.9) -- cycle;
\draw[top color=gray,bottom color=white, fill opacity =1] (90:1.9) -- (0:1.9) -- (-90:1.9) -- cycle;
\foreach \x  in {0,1,3,4,6,7,9,10}
	\draw (-\x * 30+90:1.9) -- (-\x * 30+90:2.0) (-\x * 30+60:2.4) node {$x_{\x }$};
\foreach \x  in {2}
	\draw (-\x * 30+90:1.9) -- (-\x * 30+90:2.0) (-\x * 30+60:2.7) node {$x_{\x }=x_{t_0}$};
\foreach \x  in {5}
	\draw (-\x * 30+90:1.9) -- (-\x * 30+90:2.0) (-\x * 30+60:2.4) node {$x_{\x }=x_{t_1}$};
\foreach \x  in {8}
	\draw (-\x * 30+90:1.9) -- (-\x * 30+90:2.0) (-\x * 30+60:2.7) node {$x_{\x }=x_{t_2}$};
	\foreach \x  in {11}
	\draw (-\x * 30+90:1.9) -- (-\x * 30+90:2.0) (-\x * 30+60:2.4) node {$x_{\x }=x_{t_3}$};
\foreach \x  in {0, 1, ..., 11}
	\draw [fill=black] (-\x * 30+90:1.9) circle (.1cm);
\foreach \x  in {0, 1, ..., 11}
	\draw (-\x * 30+90:1.9) -- (-\x * 30+120:1.9);


\foreach \x  in { 3}
	\draw (q1) -- (-\x * 30+120:1.9);
\foreach \x  in { 12}
	\draw (q2) -- (-\x * 30+120:1.9);
\foreach \x  in {9}
	\draw (q3) -- (-\x * 30+120:1.9);
\foreach \x  in {6}
	\draw (q4) -- (-\x * 30+120:1.9);

\end{tikzpicture}
\caption{Sketch of proof of Lemma \ref{LemPre1C}.}
\label{FigNullhomotopyC}
\end{figure}
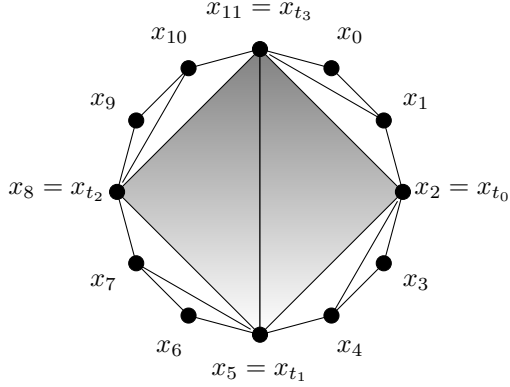
\end{proof}

\begin{Def}
 \label{DefN1C}
 [A version of Definition \ref{DefN1}]
Suppose $0 < D_1 \leq D_2$. Loop $\alpha \subset X$ is \v Cech $DC(D_1, D_2)$ isolated  if there exist two closed nested  neighborhoods $N_1 \subset N_2$ of $\alpha$, so that:
\begin{enumerate}
 \item $N_1$ and  $N_2$ are homeomorphic to $S^1 \times [0,1]$;
 \item $N_2 \supset N(N_1, D_2)$;
 \item $\di N_1$ consists of loops $\alpha_1$ and $\alpha_2$, which are at least $2 D_1$ apart from each other;
 \item $\C(\alpha_i)\simeq S^1, \forall i, \forall r<D_1$;
 \item $N_X(N_2, D_2) \setminus \Int(N_1) \DC \di N_1 $ and $N_X(N_1, D_2) \DC \alpha$.
\end{enumerate}
Loop $\alpha$ is \v Cech $DC(D)$ isolated, if it is \v Cech $DC(D,D)$ isolated. 
\end{Def}

\begin{Prop}
\label{PropRips1C}
[A version of Proposition \ref{PropRips1}]
 Suppose $D >0$ and loop $\alpha$ is \v Cech $DC(D)$ isolated. Then:
\begin{enumerate}
 \item $\C_X(N_2,D) \simeq \C_X(N_1,D) \simeq \C_X(\alpha, D)=\C_\alpha(\alpha, D)$, 
 \item $\C_X(N_2 \setminus \Int (N_1)) \simeq \C_X(\di  N_1)$, and
 \item $\C_X(\di  N_1) =\C_X(\alpha_1, D) \sqcup \C_X(\alpha_2, D)$.
\end{enumerate}
\end{Prop}

\begin{proof}
Follows by definitions, and Propositions \ref{PropDC} and \ref{PropCDC}.
\end{proof}

\begin{Thm}
 \label{ThmMain1C}
 [A version of Theorem \ref{ThmMain1}]
Suppose $X$ is a geodesic surface, $\alpha$ is a \v Cech $DC(D)$ isolated loop for some $D > |\alpha|/4$, and $G$ is a group.  Then 
 $\{H_k(\C_{\alpha}(\alpha,r),G)\}_{r<D}$ is a direct summand of  $\{H_k(\C_X(X,r),G)\}_{r<D}$ via the inclusion-induced map for all $k \geq 2$. 
 \end{Thm}

\begin{proof}
 The proof is the same as that of Theorem \ref{ThmMain1}.
\end{proof}

\begin{Def}[A version of Definition \ref{DefN2}]
 \label{DefCN2}
Suppose $0 < a \leq b$, $G$ is a group and $k\in \NN$. Subspace $Z$ of a metric space $X$ is \v Cech $DC(\langle a, b \rangle ; k, G)$ isolated if there exist two closed nested  neighborhoods $N_1 \subset N_2$ of $Z$, so that:
\begin{enumerate}
 \item $N_2 \supset N(N_1, r),  \forall  r\in \langle a, b \rangle$;
 \item  for each $r\in \langle a, b \rangle$,  condition $H_k(\C_X(Z,r),G)\neq 0$ implies that the following maps are trivial:
	\begin{enumerate}
	\item the inclusion induced  $H_k(\C_X(\di N_1,r),G) 
	\to H_k(\C_X( N_1,r),G)$ and $H_k(\C_X(\di N_1,r),G) \to 
	H_k(\C_X( X \setminus \Int(N_1),r),G)$;
	\item the boundary map 
	$$
	H_k(\C_X(X,r),G) \to H_{k+1}(\C_X(N_2 \setminus \Int 	
	(N_1),r))
	$$
	arising from the Mayer-Vietoris long exact sequence for a 	
	decomposition of $\C_X(X,r)$ into $A=\C_X(N_2,r)$ and 
	$B= 
	\C_X(X \setminus \Int (N_1),r)$;
	\end{enumerate}
\item $N_X(N_2, R) \setminus \Int(N_1) \DC \di N_1 $ and $N_X(N_1, R) \DC \alpha$ for some $R$ larger that all elements of $\langle a, b \rangle$.
\end{enumerate}
\end{Def}

\begin{Thm}[A version of Theorem \ref{ThmMain15}]
 \label{ThmCMain15}
Suppose $X$ is a metric space, $G$ is a group, $0<a\leq b$, $k \in \ZZ$, and $Z\subset X$ is \v Cech  $DC(\langle a, b \rangle; G)$ isolated.  Then 
 $\{H_k(\C_Z(Z,r),G)\}_{r\in \langle a, b \rangle}$ is a direct summand of  $\{H_k(\C_X(X,r),G)\}_{r\in \langle a, b \rangle}$ via the inclusion induced map.
 \end{Thm}

\begin{Thm}
 \label{ThmMain2C}
 [A version of Theorem \ref{ThmMain2}]
Suppose $X$ is a geodesic surface, $\alpha$ is a \v Cech $DC(D_1, D_2)$ isolated loop for some $D_1 > |\alpha|/4, D_2 \geq |\alpha|/2$, and $G$ is a group.  Assume $\alpha$ is homologous in $H_1(X,G)$ to a $G$-combination of loops of length at most $|\alpha|$, none of which intersects $N_1$. Then the following hold:
\begin{enumerate}
\item for each $r\in (|\alpha|/4, D]$ there exists a non-trivial $Q_r\in H_2(\C(X,r),G)$ so that:
\begin{enumerate}
\item  for each pair $q_1< q_2$ of parameters from $(|\alpha|/4, D]$ we have $i^G_{q_1,q_2}(Q_{q_1})=Q_{q_2}$;
\item for any $q\in (|a|/4, D]$ there exists no $q_0\leq |\alpha|/4$, for which $Q_q$ is in the image of $i^G_{q_0, q}$;
 \item If  $\alpha$ is homotopic to some shorter geodesic circle $\beta$  in $X$ and $4 q_3$ is larger than the homotopy height between $\alpha$ and $\beta$, then $i^G_{q,q_3}(Q_{q})$ is trivial for any $q\in (|a|/4, D]$.
\end{enumerate}
  
  \item If $G$ is a field and $\{H_2(\C(X,r),G)\}_{r>0}$ is $q$-tame, then the persistence $\{H_2(\C(X,r),G)\}_{r<D}$ contains as a direct summand $G_{(|\alpha|/4, w' /4)}$. 
\end{enumerate}
\end{Thm}

\begin{proof}
 The proof is essentially the same as that of Theorem \ref{ThmMain2}, using Lemma \ref{LemPre1C} and square-like formations of four points on loops instead of Lemma \ref{LemPre1} and three equidistant points. 
\end{proof}

\subsection{Closed filtrations}
\label{Sectsmth}

A technical complication arising in the case of closed filtration is the fact that deformation contractions (and even strict deformation contractions) do not necessarily induce homotopy equivalences on the corresponding closed Rips or \v Cech complexes, as demonstrated by Figure \ref{FigClosedEx}. Nonetheless, we can still detect almost the same footprint with the use of interleaving. 

Open and closed filtrations are $\eps$-interleaved for each $\eps > 0$. Hence by the standard stability results PDs have the same barcodes in both cases, with a possible change in the interval endpoint types (for example, such a change occurs in Theorem \ref{ThmAA} and \cite{ZV}). This means that the following hold in the case when a PD exists (i.e., if $G$ is a field and the persistent homology is q-tame):
\begin{enumerate}
 \item the main footprint detection results of this paper (Theorems \ref{ThmMain1}, \ref{ThmMain15}, \ref{ThmMain1}, Corollary \ref{CorDetect} and their \v Cech versions) hold for closed filtrations  modulo the endpoints, i.e., the endpoints of the detected intervals may change.
 \item closed filtrations may see an emergence or disappearance of  ephemeral summands, i.e. intervals of length $0$, as in Theorem \ref{ThmAA}. 
\end{enumerate}

While detecting ephemeral summands may be complicated in general, the understanding of geometric background allows us to detect  them in the case of geodesic circles as described by Proposition \ref{PropEph}.

\begin{Def}
Let $X$ be a metric space and $A \subset X$ a closed subspace. A map $f\colon X \to A$ is a \textbf{strict contraction} if 
\begin{itemize}
\item  $f(a)=a, \forall a\in A$;
\item $d(x,y) > d(f(x), f(y)), \forall x\in X, y\in X\setminus A$.
\end{itemize}

\end{Def}

\begin{Prop}
 \label{PropEph}
 Suppose $G$ is a group, $X$ is a geodesic space, and $\alpha \subset X$ is a geodesic circle. Suppose there exists a strict contraction $f\colon \cN(\alpha, |\alpha| \frac{l}{2l+1}) \to \alpha$. Then the inclusion $\alpha \to X$ induces an injection 
 $$
H_{2l} ( \bigvee_{|\RR|} S^{2l}, G) \cong H_{2l} \left (\cRips \left (\alpha, |\alpha| \frac{l}{2l+1} \right ), G \right ) \to   H_{2l} \left (\cRips \left (X, |\alpha| \frac{l}{2l+1} \right ), G \right ). $$
\end{Prop}

\begin{proof}
By Theorem \ref{ThmAA}, $\bigvee_{|\RR|} S^{2l}\simeq \cRips(\alpha, |\alpha| \frac{l}{2l+1})$.  
 We parameterize $\alpha$ with complex coordinates by isometrically identifying $\alpha$ with $\{z\in \CC \mid |z|=|\alpha|/(2\pi)\}$. For each $t \in [0,\frac{1}{2l+1})$, let $\beta_t=\{\exp (i(t + \frac{j}{2l+1})) \mid j\in\{0, 1, \ldots, 2l\} \}$ as defined in \cite{AA}.

In \cite[Sections 7 and 8]{AA} the basis of $H_{2l} \left (\cRips \left (\alpha, |\alpha| \frac{l}{2l+1} \right ), G \right )$ is provided in terms of cross-polytopal spheres, generated by collections of points of the form $\beta_0 \cup \beta_t$. An argument in \cite{AA} demonstrating that such a sphere is not nullhomologous in $H_{2l} \left (\cRips \left (\alpha, |\alpha| \frac{l}{2l+1} \right ), G \right )$ is based on an observation that the simplex spanned by $\beta_t$ is a maximal simplex in $\cRips \left (\alpha, |\alpha| \frac{l}{2l+1} \right)$ not appearing in any other such sphere.
In particular, a basic geometric observation shows that if for some point $x\in \alpha$ and for some $t$ condition 
 \begin{equation}
\label{Eq0}
d_\alpha(x,y) \leq |\alpha| \frac{l}{2l+1}, \quad \forall y\in \beta_t 
\end{equation}
holds, then $x\in \beta_t$. This means that no choice of a point $x\in \alpha$ and $t$ satisfies 
\begin{equation}
\label{Eq1}
d_\alpha(x,y) < |\alpha| \frac{l}{2l+1}, \quad \forall y\in \beta_t. 
\end{equation}

This implies that $\beta_t$ is also a maximal simplex in $\cRips \left (X, |\alpha| \frac{l}{2l+1} \right)$: if it was in a boundary of a simplex containing $z\notin \beta_t$, then $z\notin \alpha$ as it satisfies the condition of Equation \ref{Eq0}, and thus $f(z)$ satisfies  Equation \ref{Eq1}, a contradiction. Hence none of the  $\beta_t$ is nullhomologous in $H_{2l} \left (\cRips \left (X, |\alpha| \frac{l}{2l+1} \right ), G \right )$.
\end{proof}

\begin{Prop} [A version of Theorem \ref{ThmMain1} for closed filtrations]
\label{PropDetectC}
 Suppose $X$ is a geodesic space, $G$ is a field, $0 < a < b$, $k\in \NN$, and $\alpha$ is an $SDC( ( a, b ); k, G)$ isolated loop in $X$.  Then each interval of
 $\{H_k(\cRips(\alpha,r),G)\}_{r \in ( a, b )}$ appears in  $\{H_k(\cRips(X,r),G)\}_{r\in ( a, b )}$ via the inclusion-induced map, with the only potential modification being the left endpoints of the non-trivial intervals, which may be added.
  \end{Prop}

\begin{proof}
 By Proposition \ref{PropEph}, the statement holds for the ephemeral summands of $H_k(\cRips(\alpha, r),G)$, i.e., for the case when $k$ is even. If $k$ is odd, then by Theorem \ref{ThmAA}, $\{H_k(\Rips(\alpha,r),G)\}_{r\in ( a, b )}$ contains precisely one interval, namely 
 $$
 I=\left (\frac{k-1}{2k}, \frac{k+1}{2(k+2)}\right) \cap ( a, b ). 
 $$
Since $I$ is an open interval we can use the interleaving argument above for each value of $r\in(a,b)$, to conclude that the inclusion $\alpha \to X$ induces the same interval in $\{H_k(\cRips(X,r),G)\}_{r\in \langle a, b \rangle}$, with the possible exception of changed endpoints.  
 The right endpoint of $I$ can't be added since the triviality of a homology element in $\Rips(X,r)$, which is born before $r$, implies its triviality in $\cRips(X,r)$.
 \end{proof}

We conjecture that the left endpoints in Proposition \ref{PropDetectC} don't get added either. 

At the first sight, detecting the ephemeral summands as in Proposition \ref{PropEph} seems to be  a theoretical curiosity. However, results of \cite{AA} suggest that when approximating persistence of a geodesic circle by that of a finite subset, the ephemeral summands tend to prolong themselves to non-trivial intervals (see (5) in Section \ref{SectSample}). We plan to explore this phenomenon in geodesic spaces in future work. On the other hand, we do not expect our inability to detect the nature of the left endpoints in Proposition \ref{PropDetectC} to have a significant effect on the subsequent computations.

\section{An example of interpretation}
\label{SectSample}

In this section we present an interpretation of a computational example. From a two-dimensional sphere of radius $1$ remove an open $2$-dimensional ball of radius $1/4$ to obtain a contractible space $X$, which we equip with the geodesic metric (see the left side of Figure \ref{FigEnd}). Alternatively, $X$ can be thought of as a closed portion of a sphere of radius $1$ below the upper parallel $\alpha$ of radius $1/4$. Note that $\alpha$ is a geodesic circle in $X$ of length $\pi/2$. We select a subset $S\subset X$ consisting of $4000$ points uniformly at random, approximate the geodesic metric on $S$, and use Ripser \cite{UB} to compute the PD of the Rips filtration of a random subsample of $S$ consisting of $400$ points. The PD up to dimension $3$  was  computed by Matija \v Cufar as a part of his master's thesis \cite{Cufar} and is shown on Figure \ref{FigEnd}. It can be interpreted as follows: 

\begin{figure}[htbp]
\begin{center}
\begin{tikzpicture}[scale=.4]
    \begin{scope}
        \clip (130: 2.5) arc (180:-180: 1.6 and .3);
                \shade[ball color=blue!30!gray!20,shading angle=180] (0,0) circle (2.5);
    \end{scope}
    
    \begin{scope}
    \clip (130: 2.5) arc (-180:0: 1.6 and .3)
    arc (50:-230:2.5);
        \shade[ball color=blue!70!gray,opacity=0.90] (0,0) circle (2.5);
    \end{scope}
    to[out=200,in=-30] (130:2.5);
    \draw[red,very  thick] (130: 2.5) arc (180:-180: 1.6 and .3);

\end{tikzpicture}
\includegraphics[scale=.4]{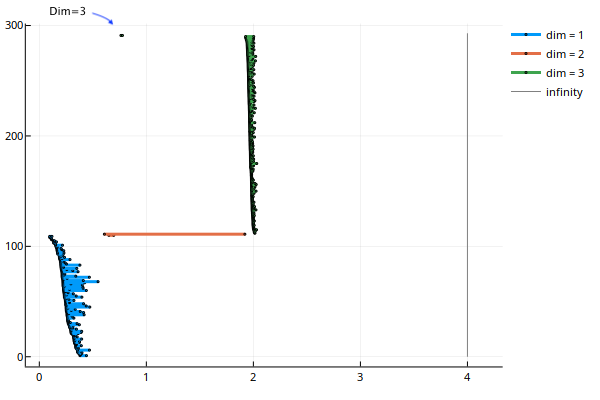}
\caption{PD described in Section \ref{SectSample}.}
\label{FigEnd}
\end{center}
\end{figure}
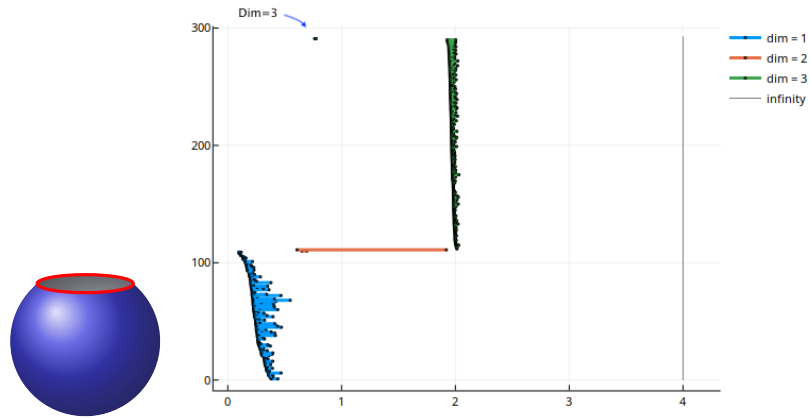

\begin{enumerate}
\item  Since $X$ is contractible, the PD should have no bars for small $r$ by \cite{Haus}. However, since we start with a discrete set, we get certain $1$-dimensional bars, which are limited to small values of $r$. See \cite{ZV1} for specific bounds on the initial $1$-dimensional bars with respect to the density of $S$. 

\item At approximately $|\alpha|/3=\pi/6$ two bars are born: 
\begin{itemize}
\item A short $3$-dimensional bar  by Theorem \ref{ThmMain1} (see the topmost feature of Figure \ref{FigEnd}), which dies at about $2|\alpha|/5=\pi/5$. The loop $\alpha$ can be located by connecting the vertices of the homology class generating this bar.
\item A long $2$-dimensional bar by Theorem \ref{ThmMain2}, which  dies at about $2 \pi/3$, which is the length of the equator divided by $3$ and thus equals the nullhomotopy height of $\alpha$.
\end{itemize}
\item Note that the long $2$-dimensional bar above is born slightly earlier than the $3$-dimensional bar. This is always the case, as generating the two-dimensional bar only requires a $2$-dimensional portion of the generator of the $3$-dimensional bar, that  spans the sample of $\alpha$.
\item A pairing of a $3$-dimensional bar with $2$-dimensional bar indicates that $\alpha$ is contractible in $X$. 
\item We speculate the other short $3$-dimensional bars are induced by other geodesic circles (i.e., equator and its rotations) in $X$. We will delve deeper into them in our future work.
\end{enumerate}

Note that, except for small values of $r$, there is essentially no noise in the PD. We are able to interpret almost all of the bars. Initial $1$-dimensional bars are unavoidable as we always start with a finite sample (discrete subset). They shorten as the density of our sample increases. The only other unmentioned bar is the short $2$-dimensional bar appearing at about the same time as the long $2$-dimensional bar. It can be explained by the effect of discretisation and the structure of the $3$-dimensional bar born at about the same time. 

During our experimentation we have generated several instances of the PD using the mentioned procedure. The obtained diagrams are qualitatively the same in all instances (and aligned with the interpretation above) with the only exception being the short isolated $3$-dimensional bar, which did not appear in all attempts due to its short length.


\begin{thebibliography}{99}

\bibitem{AA}
M. Adamaszek and Adams, H.:
\emph{The Vietoris-Rips complexes of a circle}, 
Pacific Journal of Mathematics 290-1 (2017), 1--40. 

\bibitem{AAS}
M. Adamaszek, Adams, H., and S. Reddy:
\emph{On Vietoris-Rips complexes of ellipses}, 
Journal of Topology and Analysis 11 (2019), 661-690. 

\bibitem{A3}
Adams, H., S. Chowdhury, A. Jaffe, and B. Sibanda:
\emph{Vietoris-Rips complexes of regular polygons}, 
arXiv:1807.10971. 



\bibitem{ACW}
Adams, H., Coldren, E., Willmot, S.:
\emph{The persistent homology of cyclic graphs}, 
arXiv:1812.03374.

\bibitem{ACos}
Adams, H., Coskunuzer, B.:
\emph{Geometric Approaches on Persistent Homology}, 
arXiv:2103.06408.

\bibitem{Att}
Attali, D., Lieutier, A., Salinas, D.: \emph{Vietoris-Rips complexes also provide topologically correct reconstructions of sampled shapes}. In Proceedings of the 27th annual ACM symposium on Computational geometry, SoCG '11, pages 491--500, New York, NY, USA, 2011. ACM.

\bibitem{UB} Bauer, U.: \emph{Ripser}, (2006) https://github.com/Ripser/ripser. 

\bibitem{Comb}
Cencelj, M., Dydak, J., Vavpeti\v c, A., Virk, \v Z: 
\emph{A combinatorial approach to coarse geometry},
Topology and its Applications 159(2012), 646--658.

\bibitem{CS}Chambers, E.W., de Silva, V., Erickson, J., Ghrist, R.:  \emph{Rips complexes of planar
point sets}. Discrete \and Computational Geometry, 44(1):75--90, 2010.

\bibitem{CL} Chambers, E.W., Letscher, D.: \emph{On the height of a homotopy}, Proceedings of the 21st Canadian Conference on Computational Geometry, 2009, pp. 103--106.

\bibitem{ChaObs} Chazal, F., Crawley-Boevey, W., de Silva, V.:
\emph{The observable structure of persistence modules}, Homology, Homotopy and Applications (2016) 18(2): 247 --265.

\bibitem{Cha2} Chazal, F., de Silva, V., Oudot, S.:
\emph{Persistence stability for geometric complexes}, Geom. Dedicata (2014) 173: 193. 

\bibitem{Cufar}
\v Cufar, M.:
\emph{Ra\v cunanje enodimenzionalne vztrajne homologije v geodezi\v cni metriki}, Ms Thesis, University of Ljubljana, 2020, 
\newline \textrm{https://repozitorij.uni-lj.si/IzpisGradiva.php?lang=eng\&id=114071}.

\bibitem{Dra02}
Dranishnikov, A.:
\emph{
Anti-\v{C}ech approximation in coarse geometry},
Preprint, Institut des Hautes \'{E}tudes Scientifiques,
  Bures-sur-Yvette, France, 2002.

\bibitem{DySe78}
Dydak, J., Segal, J.: 
\emph{Shape Theory.  An Introduction.}
Springer-Verlag, Berlin, Germany, 1978.

\bibitem{EW} Edelsbrunner, H., Wagner, H.: \emph{Topological data analysis with Bregman divergences}. In ``Proc. 33rd Ann. Sympos. Comput. Geom., 2017'', 39:1--39:16

\bibitem{Fro}
Frosini, P.:
\emph{Metric homotopies},
Atti del Seminario Matematico e Fisico dell'Universit\` a di Modena, XLVII, 271--292 (1999).

\bibitem{7A} 
Gasparovic, E., Gommel, M., Purvine,  E., Sazdanovic,  R., Wang,  B., Wang, Y., Ziegelmeier, L.:
\emph{A Complete Characterization of the $1$-Dimensional Intrinsic \v Cech Persistence Diagrams for Metric Graphs},
 In: Chambers E., Fasy B., Ziegelmeier L. (eds) Research in Computational Topology. Association for Women in Mathematics Series, vol 13. Springer, Cham

\bibitem{LS2}
Gornet, R., Mast, M.B.: \emph{The length spectrum of Riemannian two-step nilmanifolds}. Annales scientifiques de l'\' Ecole Normale Sup\' erieure, Serie 4, Volume 33 (2000) no. 2, pp. 181--209.

\bibitem{Hat}Hatcher, A.: \emph{Algebraic topology}. Cambridge University Press, Cambridge, 2002.

\bibitem{Haus} Hausmann, J.-C.: \emph{On the Vietoris-Rips complexes and a cohomology theory for metric spaces.}
Annals of Mathematics Studies, 138:175--188, 1995.

\bibitem{Hopf}
Hopf, H.: \emph{\" Uber die Abbildungen der dreidimensionalen Sph\" are auf die Kugelfl\" ache}, Math. Ann.  104: 637--665, 1931.

\bibitem{LengthSpec}
O. Labl\' ee: \emph{Spectral Theory in Riemannian Geometry}, European Mathematical Society, 2015.

\bibitem{Lat}
Latschev, J.: \emph{Rips complexes of metric spaces near a closed Riemannian manifold.} Archiv der
Mathematik, 77(6):522--528, 2001.

\bibitem{Memoli}
Lim, S., Memoli, F., Okutan,  O.B.:
\emph{Vietoris-Rips Persistent Homology, Injective Metric Spaces, and The Filling Radius}, 
arXiv:2001.07588.

\bibitem{W}
Niyogi, P., Smale, S., Weinberger, S.: \emph{Finding the homology of submanifolds with high confidence from random samples.} Discrete Comput. Geom., 39:419--441, March 2008.

\bibitem{ZV} 
Virk, \v Z:
\emph{1-Dimensional Intrinsic Persistence of geodesic spaces}, Journal of Topology and Analysis 12 (2020), 169--207.

\bibitem{ZV1}
Virk, \v Z:
\emph{Approximations of $1$-Dimensional Intrinsic Persistence of Geodesic Spaces and Their Stability},
Revista Matem\' atica Complutense
(2019) 32: 195--213.

\bibitem{ZV2}
Virk, \v Z:
\emph{Rips complexes as nerves and a Functorial Dowker-Nerve Diagram},
	Mediterr. J. Math. 18,  58(2021).
\end{thebibliography}
\end{document}